\documentclass[11pt]{article}
\usepackage{amsmath,amsthm}
\usepackage{amsfonts}
\usepackage{amssymb}
\usepackage{graphicx}
\usepackage{authblk}
\usepackage{color}
\usepackage{epstopdf}
\usepackage[symbol]{footmisc}
\usepackage[table]{xcolor}
\usepackage{multirow}
\usepackage{hhline}
\usepackage{titlesec}

\newlength{\spacing}
\newtheorem{theorem}{Theorem}[section]
\newtheorem{lemma}{Lemma}[section]

\newtheorem{corollary}{Corollary}[section]
\newtheorem{remark}{Remark}[section]

\numberwithin{equation}{section}
\renewenvironment{proof} {{\bf Proof.} \ignorespaces}
{\par\medskip}

\title{WENO interpolations and reconstructions using data bounded polynomial approximation}
\author[1]{Sabana Parvin }
\author[2]{Ritesh Kumar Dubey }
\affil[*]{Research Institute \& Department of Mathematics, SRM Institute of Science and Technology, Chennai, India}
\affil[1]{\textit {sabanaps@srmist.edu.in }}
\affil[2]{\textit {riteshkd@srmist.edu.in}}

\date{}
\begin{document}

\maketitle

\begin{abstract} 
\par This work characterizes the structure of third and forth order WENO weights by deducing data bounded condition on third order polynomial approximations. Using these conditions, non-linear weights are defined for third and fourth order data bounded weighted essentially non-oscillatory (WENO) approximations. Computational results show that data bounded WENO approximations for smooth functions achieve required accuracy and do not exhibit overshoot or undershoot for functions with discontinuities and extrema. Further with suitable weights, high order data-bounded WENO approximations are proposed for WENO schemes. 
\end{abstract} 


\section{ Introduction}
Solution of partial differential equations (PDEs) like hyperbolic conservation laws and convection dominated problems admits strong irregular profiles (large jumps,discontinuities) along with complicated smooth structures which make numerical treatment of such solution challenging. It is worth mentioning that the visualization of the solution of a PDE model is an essential aid to comprehend the physical problem. The goodness of any numerical scheme for these PDEs relies on the accuracy and non-oscillatory nature of approximated numerical fluxes at cell interfaces. Also for numerical stability, scheme should be data-bounded to ensure the bounded growth of the numerical solutions \cite{laney}. This consequently depends on boundededness of the  polynomial approximation used in underlying numerical scheme. Thus high-order well-behaved polynomial approximation techniques that can interpolate both smooth and discontinuous function data without introducing “Gibbs-like” oscillations are much required. In other words, a high order polynomial approximations should respect the physical properties like boundedness and shape of given data \cite{berzins2007adaptive}. In fact, data bounded polynomial approximations which are bounded by the minimum and maximum data values may also help to preserve positivity in the solution \cite{berzins2010nonlinear}. {\it Throughout this text, until stated, the term "data bounded or bounded polynomial approximations" represents that it is bounded by the minimum and maximum of data values used to construct underlying polynomial}. 

\par Among high order polynomial approximation techniques essentially non-oscillatory (ENO) and the weighted ENO (WENO) approximations achieved phenomenal popularity due to their ability to approximate the discrete data corresponding to strong discontinuities in an essentially non-oscillatory manner while retaining high-order accuracy in smooth data regions. 
Though ENO or WENO procedures are generic function approximation techniques, they are mainly developed in the context of numerical schemes for hyperbolic conservation laws (HCL). 

\par Harten et. al \cite{harten1987uniformly} first proposed Essentially non-oscillatory(ENO) schemes which later got efficiently implemented by Shu and Osher in \cite{shu1988efficient, shu1989efficient}. ENO procedure is based on the idea of smoothest stencil selection out of many candidate stencil. In \cite{liu1994weighted}, Liu et al., proposed the WENO procedure which uses convex combination of all the candidate stencils unlike choosing the smoothest one in the ENO techniques. Later, Jiang and Shu in \cite{jiang1996efficient} introduced the most-celebrated finite-difference WENO framework popularly named as WENO-JS scheme by modifying the smoothness measurement and extended the scheme up to fifth-order accuracy. In \cite{henrick2005mapped} Henrick et al.,  proposed a mapping function to improve optimal order of accuracy at critical points as WENO-JS schemes fail to achieve the formal order of accuracy at critical points. Borges et al., \cite{borges2008improved} developed new WENO weights by using global smoothness measurement for fifth-order WENO scheme and named it as WENO-Z scheme. Further many versions of WENO schemes are developed in \cite{castro2011high,ha2013improved,rathan2018,DON2013347,xiaoshuai2015high,biswas2020eno,fan2014new,wu2016new,xu2018improved, shu1990numerical}. General information about WENO framework can be found in the surveys by Shu \cite{shu2009high} and Zhang et al. \cite{abgrall2016handbook}. Generally, WENO schemes are based upon WENO reconstruction procedure which generates a piecewise polynomial approximations of a function from given set of its cell averages. WENO reconstruction using polynomials from cell averages is equivalent to the interpolation of point values of the primitive polynomial \cite{shu2009high}. Therefore, WENO methods (algorithms) can also be formulated as an interpolation technique \cite{shu2009high}. In WENO methods the use of high order polynomials are attractive due to their potential of yielding high accurate approximation of cell interface values.

\par Despite of the substantial body of work and developments on WENO methods, there are still challenges associated with the stability analysis of these methods. Some of the work carried out in terms of stabilities of WENO schemes are  \cite{yamaleev2009high,balsara2000monotonicity,zhang2011maximum} and analysis of a peculiar sign stability property \cite{fjordholm2013eno} is done for third  order WENO reconstruction in \cite{fjordholm2016sign,biswas2018low}. However, to the best of authors knowledge the stability of WENO schemes in terms of boundedness of the solution or associated WENO approximations are not carried out so far.

\par {\it The main aim of this work is to establish conditions on data-boundedness of three point polynomial approximation and further construct high order data-bounded WENO interpolations and WENO reconstructions at cell interfaces. }

The rest of the paper is organized as follows: In Section \ref{sec-2} the data-boundedness conditions for a three-point polynomial approximations is established. This is done by analyzing the conditions on non-linear weights such that the three point polynomial approximation is bounded. These conditions are provided in terms of smoothness parameters which are ratio of consecutive gradients. Section \ref{sec-3} provides suitable non-negative weights for data-bounded WENO approximations along with computational verification for the accuracy and data-boundedness of the WENO approximations. Section \ref{sec-4} briefly discuss the application of data-bounded WENO approximations to construct WENO schemes for hyperbolic conservation laws  following author's work \cite{parvin2021new}, whereas other possible uses are mentioned in \ref{sec-5}. Section \ref{sec-6} contains the concluding remarks.
 \begin{remark}
 	It is a rudimentary proof due to Harten (in \cite{zhang2011maximum}) that a {\bf data bounded} approximation of the solution of hyperbolic PDE's in the sense of associated maximum principle can be at most second order accurate. Note that,
 	$$\mbox{Data bounded solution} \Rightarrow \mbox{Bounded solution,  but converse is not true.}$$ 
 	We {\bf emphasize} that, {\bf data bounded WENO reconstructions} proposed in this work when applied to WENO schemes for hyperbolic conservation laws eventually ensure that the computed solution of underlying PDE is  {\bf bounded}. It does not guarantee for solution to be {\bf data bounded} in the sense of maximum principle. 
 	Thus, following Theorem \ref{first-theorem} about bounded polynomial approximation or statements like third order data bounded WENO reconstruction do not contradict with the restrictive proof in \cite{zhang2011maximum} on the accuracy of data bounded solution of hyperbolic PDE.
 \end{remark}
  
\section{Data-bounded three-point polynomial approximation}\label{sec-2}
Let the set of  discrete data values $\{v_j\}_{j=i-1}^{j=i+1}$ be given for a bounded function $v(x)$ at evenly distributed spatial points $\displaystyle \{x_j\}_{j=1-1}^{j=i+1}$ with fixed grid spacing $\Delta x$. These data values can be either point values $v_j= v(x_j)$ or cell-average values $\overline v_j $ of $v(x)$ in cell $I_j =[x_{j-\frac{1}{2}}, x_{j+\frac{1}{2}}]$ i.e.,  $\displaystyle \overline v_j =\frac{1}{\Delta x}\int_{x_{j-\frac{1}{2}}}^{x_{j+\frac{1}{2}}} v(x) dx$.  Consider the following three-point polynomial approximation, 
\begin{equation}\label{poly-approx}
\hat{v}(x)= \sum_{p=0}^{1} \alpha_p \hat{v}^{p}(x)=\alpha_0 \hat{v}^{0}(x)+\alpha_1 \hat{v}^{1}(x)
\end{equation}
where weights $\alpha_i \in \mathbb{R}, i =0,1$ are such that  $\alpha_1=1-\alpha_0$ and  $\hat{v}^p$ are the following linear interpolating polynomial using data $\{v_{i+p-1},v_{i+p}\}, \; p=0,1$ in stencils $S_p(i):=\{x_{i+p-1},x_{i+p}\}$
\begin{equation}\label{approxes}
\begin{aligned}
	\hat{v}^{0}(x)&=v_{i-1}+\frac{v_{i}-v_{i-1}}{\Delta x}(x-x_{i-1})\\
	\hat{v}^{1}(x)&=v_{i}+\frac{v_{i+1}-v_{i}}{\Delta x}(x-x_{i})
\end{aligned}
\end{equation}
Define the quantities $L_i^{+}=\frac{1}{1-r_{i}^{+}},~L_i^{-}=\frac{1}{1-r_{i}^{-}}$, where $r_i^{\pm}$ is the smoothness parameter defined as
\begin{equation}\label{SP}
r_i^{\pm}=\frac{\Delta_{\mp}v_i}{\Delta_{\pm}v_i} \in \mathbb{R}\cup\{\pm \infty\}
\end{equation}
Also let  $m_i=\min\{v_{i-1},v_{i},v_{i+1}\}\; \mbox{and}\; M_i=\max\{v_{i-1},v_{i},v_{i+1}\}$ then boundedness conditions of the polynomial approximation \eqref{poly-approx} in terms of weight $\alpha_0$ and smoothness parameter $r_i$, are given by the following theorem.
\begin{theorem}\label{first-theorem} 
The polynomial approximation $\hat{v}(x), x \in [x_{i-1},x_{i+1}]$ \eqref{poly-approx} is {\bf data bounded} i.e., $m_i\leq \hat{v}(x)\leq M_i$ provided
\begin{itemize}
	\item [(a)] $K^{\pm}_1\le \alpha_0(x-x_i) \le K^{\pm}_2$, for $r_{i}^{\pm}\geq1$
	\item [(b)] $K^{\pm}_2\le \alpha_0(x-x_i) \le K^{\pm}_1$, for $r_{i}^{\pm}\in [0,1)$
	\item [(c)] $K^{\pm}_2\le \alpha_0(x-x_i) \le K^{\pm}_3$, for $r_{i}^{\pm}\in [-1,0]$
	\item [(d)] $K^{\pm}_1\le \alpha_0(x-x_i) \le K^{\pm}_3$, for $r_{i}^{\pm}\le -1$
\end{itemize}
where the bounds for data-bounded region are \\
$\begin{array}{llrl}
	K_1^+ =& L_i^{+} \left(r_{i}^{+} \Delta x+(x-x_i)\right), & K_1^{-} =& 	L_i^{-}\left(r_{i}^{-}\left(\Delta x-(x-x_i)\right)\right) \\
	K_2^{+} =& L_i^{+}(x-x_{i+1}), & K_2^{-}=&-L_{i}^{-}\left(\Delta x+(x-x_i)r_i^{-}\right) 
	\\
	K_3^{+}=& L_i^{+}(x-x_{i}), & K_3^{-} =&-L_{i}^{-}\left((x-x_i)r_i^{-}\right) 
\end{array}$
\\  
\end{theorem}
\begin{proof}
Consider the case when smoothness parameter is defined as $r_i^{+}$. The proof for $r_i^{-}$ follows analogously.\\
{\bf Case I: Monotone Data $r_i^+\ge 0$:}  By rearranging the terms, polynomial approximation \eqref{poly-approx} can be written in the following forms,
\begin{subequations}
\begin{equation}\label{form1}
	\hat{v}(x)=v_{i-1}+(1-\alpha_0)\Delta_{-}v_{i}+\frac{\alpha_0}{\Delta x}(x-x_{i-1})\Delta_{-}v_{i}+\frac{\alpha_1}{\Delta x}(x-x_{i})\Delta_{+}v_{i}
\end{equation}
\begin{equation}\label{form2}
	\hat{v}(x)=v_{i+1}-\alpha_0\left(1-\frac{x-x_{i-1}}{\Delta x}\right)\Delta_{-}v_{i}-\left(1-\frac{\alpha_1}{\Delta x}(x-x_{i})\right)\Delta_{+}v_{i}
\end{equation}
\begin{equation}\label{form3}
	\hat{v}(x)=v_{i}-\alpha_0\left(1-\frac{x-x_{i-1}}{\Delta x}\right)\Delta_{-}v_{i}+\frac{\alpha_1}{\Delta x}(x-x_{i})\Delta_{+}v_{i}
\end{equation}
\end{subequations}
Let data be monotonically increasing i.e., $m_i=v_{i-1}\le v_{i}\le v_{i+1}=M_i$ then it follows from \eqref{form1},
\begin{eqnarray}\nonumber
\hat{v}(x)&\ge& v_{i-1},~\text{if} ~(1-\alpha_0)\Delta_{-}v_{i}+\frac{\alpha_0}{h}\Delta_{-}v_{i}(x-x_{i-1})+\frac{\alpha_1}{h}\Delta_{+}v_{i}(x-x_{i})\ge 0
\end{eqnarray}
or
\begin{equation}\label{comb1}
\hat{v}(x)\geq m_i ~~~~\text{if}~\alpha_0(x-x_i)(1-r_{i}^{+})\le r_{i}^{+}\Delta x+(x-x_i)
\end{equation}
Similarly from \eqref{form2}
\begin{eqnarray}\label{comb2}
\hat{v}(x)&\le& v_{i+1}~(=M_{i}),~\text{if}~ \alpha_0(x-x_i)(1-r_{i}^{+})\ge x-x_{i+1}
\end{eqnarray}
together with \eqref{comb1}  and \eqref{comb2};
\begin{equation}\label{DB1}
m_{i}\le \hat{v}(x)\le M_{i},\; \mbox{provided}\;(x-x_{i+1})\le \alpha_0(x-x_i)(1-r_{i}^{+}) \le \left(r_{i}^{+}\Delta x+(x-x_i)\right).
\end{equation}
Conditions (a) and (b) of theorem  follows from the conditional compound inequality in \eqref{DB1}. Similarly these conditions for monotonically decreasing data i.e., $M_i = u_{i-1}\geq u_i \geq u_{i+1} = m_i$ can be obtained.\\ 
{\bf Case 2: Non Monotone data $r_{i}^{+}\le 0$:} It implies that given data consists extrema. Here calculations are given only for the case when $v_i$ is minima as similar calculations follows for the case $v_i$ is maxima. In case of minima ,    
$$v_i\le\min(v_{i-1},v_{i+1}) :\Rightarrow  m_{i}=v_{i}, M_{i}=\max(v_{i+1}, v_{i-1}),\mbox{and}\;\Delta_{-}v_{i}\le0,~\Delta_{+}v_{i}\ge0.$$
Following sub-cases can occur 
\begin{itemize}
\item[({\bf i})] $v_{i-1}\leq v_{i+1} \Rightarrow M_{i}=v_{i+1}~\mbox{and}~v_{i}-v_{i-1}\geq v_{i}-v_{i+1}\Rightarrow r_{i}^{+}\geq -1.$
Therefore $r_{i}^{+}\in[-1,0]$ and $1\le (1-r_{i}^{+})\leq2.$ It follows from the equation \eqref{form3}
\begin{eqnarray}\nonumber
	\hat{v}(x)&\ge & v_{i},~\text{if} ~\alpha_0\left(1-\frac{x-x_{i-1}}{\Delta x}\right)\Delta_{-}v_{i}-\frac{\alpha_1}{\Delta x}(x-x_{i})\Delta_{+}v_{i}\le 0\\\label{comb3}
	&\ge& m_i ~\text{if}~ \alpha_0(x-x_i)\le \frac{x-x_i}{1-r_{i}^{+}}
\end{eqnarray}
Similarly using \eqref{form2}
\begin{eqnarray}\label{comb4}
	\hat{v}(x)&\le& v_{i+1}~(=M_{i}),~\text{if}~ \alpha_0(x-x_i)\ge \frac{x-x_{i+1}}{1-r_{i}^{+}}
\end{eqnarray}
Equation \eqref{comb3} and \eqref{comb4} establish condition (c) of the theorem as,
\begin{equation}\nonumber
	m_{i}=v_{i}\le \hat{v}(x)\le v_{i+1}=M_{i},\; \mbox{provided}\;\frac{x-x_{i+1}}{1-r_{i}^{+}}\le \alpha_0(x-x_i) \le  \frac{x-x_i}{1-r_{i}^{+}}.
\end{equation}
\item [({\bf ii})]  $v_{i-1}\geq v_{i+1} \Rightarrow M_{i}=v_{i-1}$ and $ v_{i}-v_{i-1}<v_{i}-v_{i+1}\Rightarrow r_{i}^{+}\leq -1$ and $(1-r_{i}^{+})\geq 2.$
Now from \eqref{form1} we get, 
\begin{eqnarray}\label{comb5}
	\hat{v}(x)&\le& v_{i-1}~(=M_{i}),~\text{if}~ \alpha_0(x-x_i)\le \frac{r_{i}^{+}\Delta x+(x-x_i)}{(1-r_{i}^{+})}
\end{eqnarray}
Equation \eqref{comb3} and \eqref{comb5} implies condition (d) of the theroem i.e., 
\begin{equation}\nonumber
	m_{i}=v_{i}\le \hat{v}(x)\le v_{i-1}=M_{i},\; \mbox{provided}\; \frac{r_{i}^{+}\Delta x+(x-x_i)}{(1-r_{i}^{+})}\le \alpha_0(x-x_i) \le  \frac{x-x_i}{1-r_{i}^{+}} .
\end{equation}

\end{itemize}

\end{proof}
\subsection{Conditions for data-bounded approximations at cell-interfaces:}\label{subsec-2.1}
Note that the conservative numerical approximation of PDE's requires approximations at cell interfaces $x_{i\pm\frac{1}{2}}$. 
Following lemmas follows from Theorem \eqref{first-theorem}
\subsubsection{At cell interface $\mathbf{x_{i+\frac{1}{2}}}$:}
The polynomial approximation \eqref{poly-approx} at $x_{i+\frac{1}{2}}$ can be written as
\begin{equation}\label{reconst1}
\hat v_{i+\frac{1}{2}}= \beta_0 \hat v^{0}_{i+\frac{1}{2}} + \beta_1 \hat v^{1}_{i+\frac{1}{2}},
\end{equation}
where $\beta_0,~\beta_1$ are non-linear weights such that  $\beta_1=1-\beta_0$ and from \eqref{approxes}, $\hat v^{0}_{i+\frac{1}{2}},~\hat v^{1}_{i+\frac{1}{2}}$ can be written as
\begin{equation}\label{subreconst1}
\begin{aligned}
	\hat{v}^{0}_{i+\frac{1}{2}}=\frac{3}{2}v_i-\frac{1}{2}v_{i-1},\\
	\hat{v}^{1}_{i+\frac{1}{2}}=\frac{1}{2}v_i+\frac{1}{2}v_{i+1}.
\end{aligned}
\end{equation}

\begin{lemma}\label{lemma1}
The polynomial approximation  \eqref{reconst1} using the given data $\left(x_j,v_j\right)_{j=i-1}^{j=i+1}$ is bounded i.e., $m_i\le \hat v_{i+\frac{1}{2}}\le M_i $ provided
\begin{itemize}
	\item [(a)] $(1+2r_i^{+})L_i^{+}\le \beta_0 \le -L_i^{+}$, for $r_{i}^{+}\ge 1$
	\item [(b)] $-L_i^{+}\le \beta_0 \le (1+2r_i^{+})L_i^{+}$, for $r_{i}^{+}\in [0,1)$
	% 		$-\infty\le \beta_0 \le \infty$ for $r_{i}^{+}=1$
	\item [(c)] $-L_i^{+}\le \beta_0 \le L_i^{+}$, for $r_{i}^{+}\in [-1,0]$
	\item [(d)] $(1+2r_i^{+})L_i^{+}\le \beta_0 \le L_i^{+}$, for $r_{i}^{+}\le -1$
\end{itemize}
\end{lemma}	
The WENO approximation at $x_{i+\frac{1}{2}}$ using \eqref{reconst1} is defined as convex combination of $\hat{v}^{0}_{i+\frac{1}{2}}$ and $\hat{v}^{1}_{i+\frac{1}{2}}$ i.e.,  the non-linear weights $\beta_0\ge 0,\beta_1\ge 0;~\beta_0+\beta_1=1$. Under these condition on weights, we have  
\begin{corollary}\label{corrolary1}
The weighted essentially non oscillatory (WENO) approximation $\hat v_{i+\frac{1}{2}}$ using \eqref{reconst1} is  data-bounded $(m_i\le \hat v_{i+\frac{1}{2}}\le M_i)$ under the condition
\begin{equation}\label{U.B-K}
	% 		K=\left\{\begin{array}{cc} 1 & if~r_{i}\in [0,2], \\
	% 		\frac{Sign(r_i)}{r_i-1} & elsewhere. \end{array}\right.
	% 		\textbf{or}~
	0\le \beta_0 \le K;~\text{where}
	~K=\min\left(1, \frac{sgn(r_i^{+})}{r_i^{+}-1}\right),~~sgn(r_i^{+})=
	\begin{cases}
		~~1~~ \text{if}~~r_i^{+}>0,\\
		-1~~ \text{if}~~r_i^{+}\leq 0.
	\end{cases}
\end{equation}
%	and $sgn(r_i^{+})=
%	\begin{cases}
%	~~1~~ \text{if}~~r_i^{+}>0,\\
%	-1~~ \text{if}~~r_i^{+}\leq 0.
%	\end{cases}$
\end{corollary}
In Figure \ref{fig:RDBregion} geometric interpretation of the relation between non-linear weight $\beta_0$ and smoothness parameter $r_{i}^{+}$ for data bounded approximation at cell interface $x_{i+\frac{1}{2}}$ is given in view of Lemma \ref{lemma1} and Corollary \ref{corrolary1}.
\begin{figure}[htb!] 
\hspace{-1.2cm}
\begin{tabular}{cc}
	% 		\hspace{-1.2cm}
%	\includegraphics[scale=0.55]{./pdf_figure//rightDB.pdf}
	\includegraphics[scale=0.55]{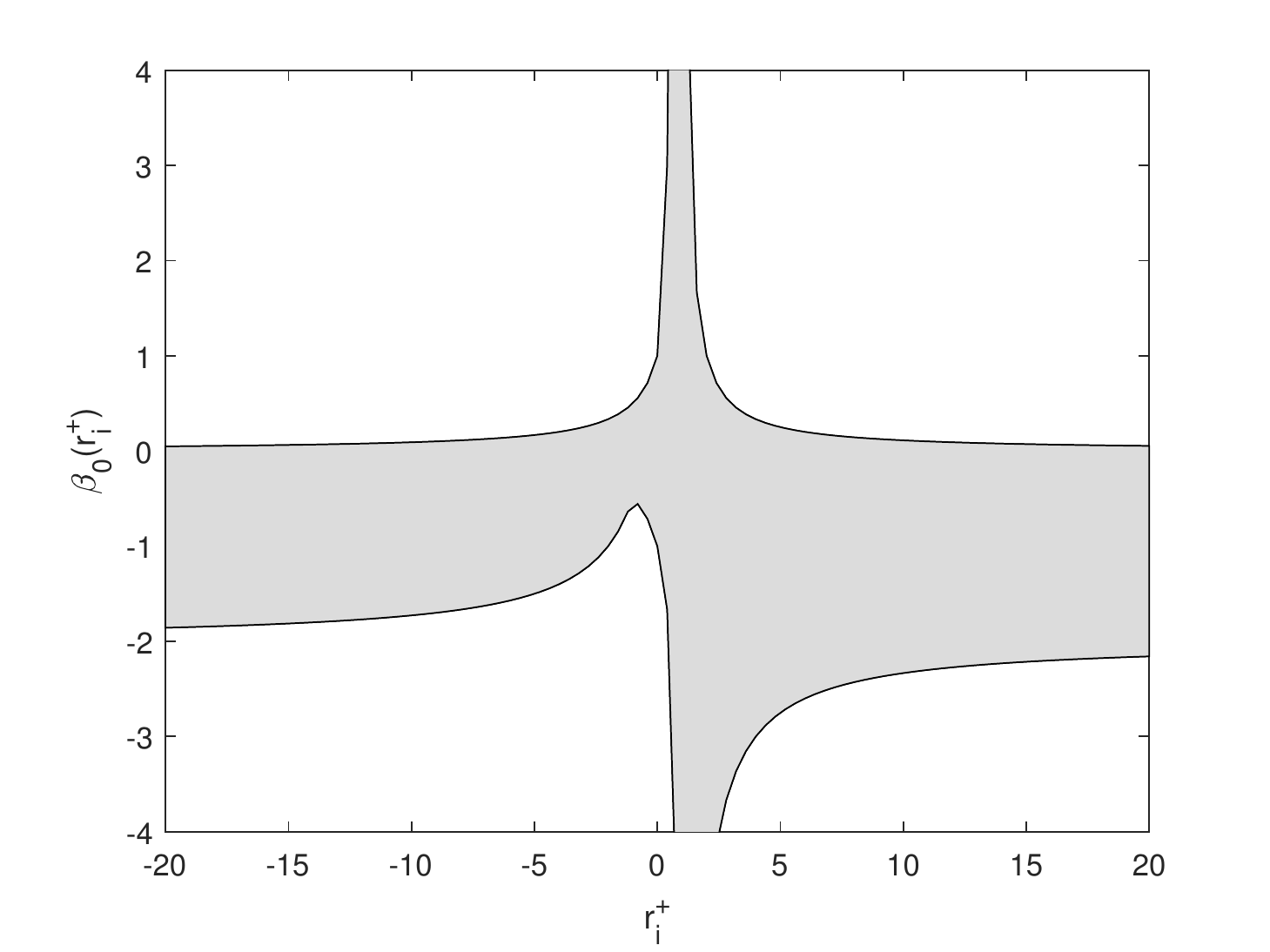} &
	\includegraphics[scale=0.55]{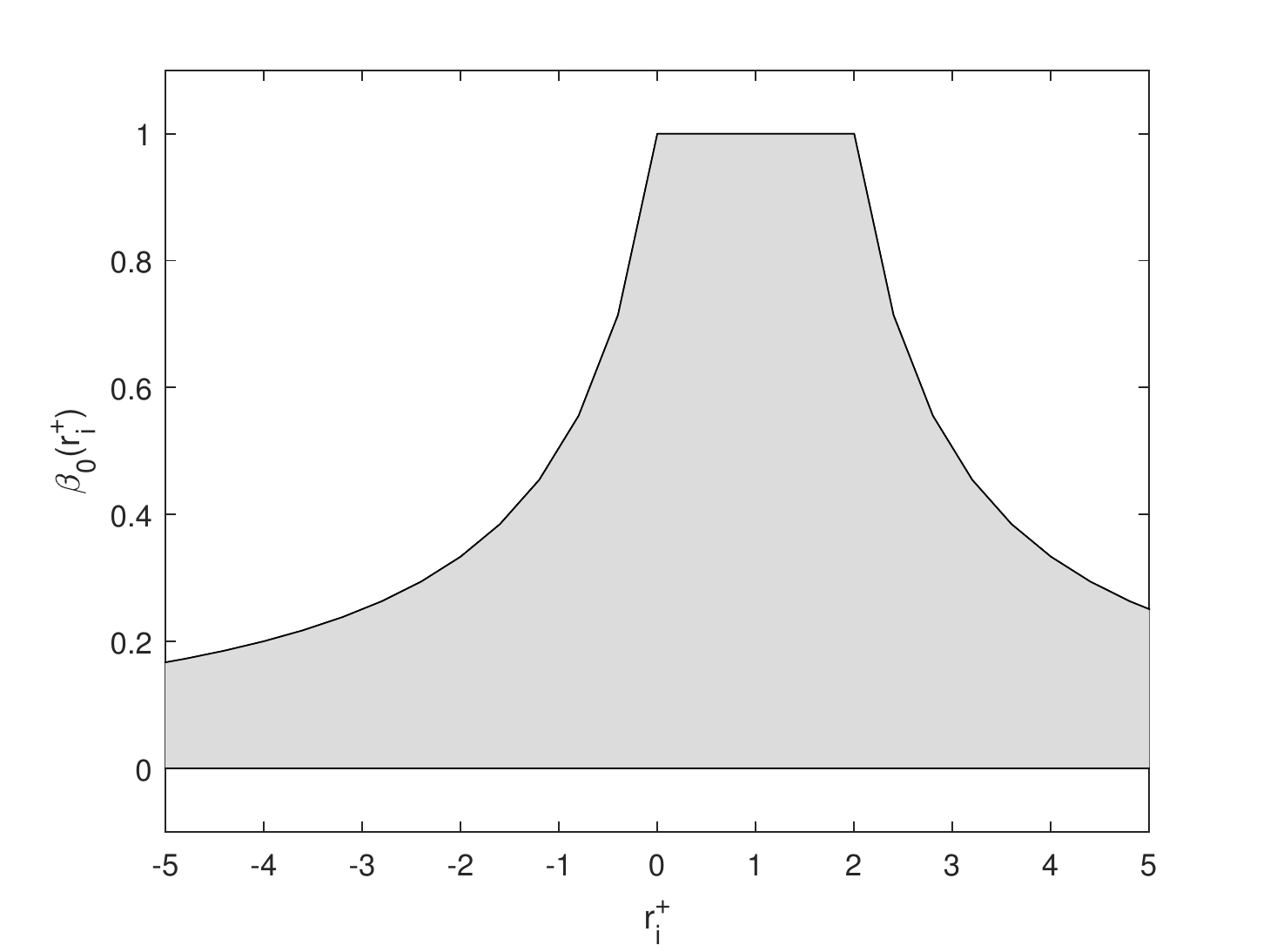} \\
	\textbf{a}&\textbf{b}
\end{tabular}
\caption{Region for data-bounded approximation using \eqref{reconst1} at $x_{i+\frac{1}{2}}:$ \textbf{(a)} Any approximation, \textbf{(b)} WENO approximation.}	
\label{fig:RDBregion}	
\end{figure}

\subsubsection{At cell interface $\mathbf{x_{i-\frac{1}{2}}}$:} The polynomial approximation $\hat{v}(x)$ \eqref{poly-approx} at  $x_{i-\frac{1}{2}}$ can be written in the form
	\begin{equation}\label{reconst2}
		\hat v_{i-\frac{1}{2}}= \mu_0 \hat v^{0}_{i-\frac{1}{2}} + \mu_1 \hat v^{1}_{i-\frac{1}{2}},
	\end{equation}
	where $\mu_0,~\mu_1$ are non-linear weights ;  $\mu_1=1-\mu_0$. And from \eqref{approxes} $\hat v^{0}_{i-\frac{1}{2}},~\hat v^{1}_{i-\frac{1}{2}}$ are as follows
	\begin{equation}\label{subreconst2}
		\begin{aligned}
			\hat{v}^{0}_{i-\frac{1}{2}}=\frac{1}{2}v_i+\frac{1}{2}v_{i-1}\\
			\hat{v}^{1}_{i-\frac{1}{2}}=\frac{3}{2}v_i-\frac{1}{2}v_{i+1}
		\end{aligned}
	\end{equation}
	\begin{lemma}\label{lemma2}
		The polynomial approximation $\hat v_{i-\frac{1}{2}}$ \eqref{reconst2} using the given data $\left(x_j,v_j\right)_{j=i-1}^{j=i+1}$ is bounded i.e.,  $m_i\le \hat v_{i-\frac{1}{2}}\le M_i $ provided
		\begin{itemize}
			\item [(a)] $(2-r_i^{-})L_i^{-} \le \mu_0 \le -3r_i^{-}L_i^{-}$, for $r_{i}^{-}\ge 1$
			\item [(b)] $-3r_i^{-}L_i^{-} \le \mu_0 \le (2-r_i^{-})L_i^{-}$, for $r_{i}^{-}\in [0,1)$
			% 				$-\infty\le \mu_0 \le \infty$ for $r_{i}^{-}=1$
			\item [(c)] $-r_i^{-}L_i^{-} \le \mu_0 \le (2-r_i^{-})L_i^{-}$, for $r_{i}^{-}\in [-1,0]$
			\item [(d)] $-r_i^{-}L_i^{-} \le \mu_0 \le -3r_i^{-}L_i^{-}$, for $r_{i}^{-}\le -1$
		\end{itemize}
	\end{lemma}
The WENO approximation at $x_{i-\frac{1}{2}}$ using \eqref{reconst2} is defined as convex combination of $\hat{v}^{0}_{i-\frac{1}{2}}$ and $\hat{v}^{1}_{i+\frac{1}{2}}$ i.e.,  the non-linear weights $\mu_0\ge 0,\mu_1\ge 0;~\mu_0+\mu_1=1$. Under these condition on weights, we have

	\begin{corollary}\label{corrolary2}
		The weighted essentially non oscillatory (WENO) approximation $\hat v_{i-\frac{1}{2}}$ using \eqref{reconst2} is data-bounded $(m_i\le \hat v_{i-\frac{1}{2}}\le M_i)$ under the condition 
		\begin{equation}\label{L.B-J}
			% 			J=\left\{\begin{array}{cc} \frac{1-2r_i}{1-r_i} & ~~~~if~r_{i}\in [0,\frac{1}{2}), \\
			% 			0 & if~r_{i}\ge \frac{1}{2}, \\
			% 			\frac{1}{1-r_i} & ~elsewhere. \end{array}\right.
			% 			\textbf{or}~
			{J\le \mu_0 \le 1};~\text{where} ~J=\max\left(0,\min\left(\frac{2-r_i^{-}}{1-r_i^{-}}, \frac{-r_i^{-}}{1-r_i^{-}}\right)\right)
		\end{equation}
	\end{corollary}
	In Figure \ref{fig:LDBregion} geometric interpretation of the relation between non-linear weight $\mu_0$ and smoothness parameter $r_{i}^{-}$ for data bounded approximation at cell interface $x_{i-\frac{1}{2}}$ is given in view of Lemma \ref{lemma2} and Corollary \ref{corrolary2}.
	\begin{figure}[htb!] 
		\hspace{-1.2cm}
		\begin{tabular}{cc}
			% 		\hspace{-1.2cm}
			\includegraphics[scale=0.55]{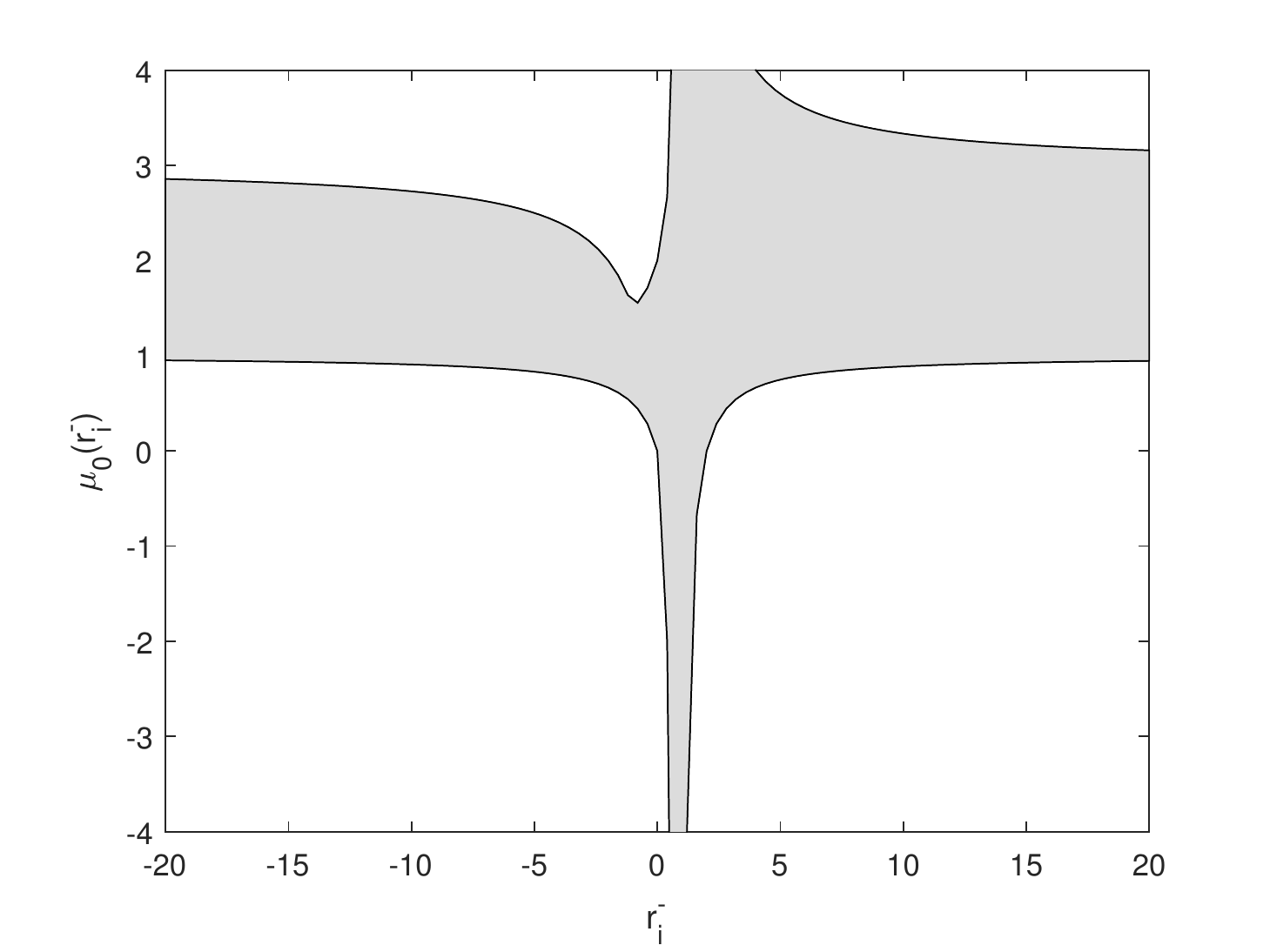} &
			\includegraphics[scale=0.55]{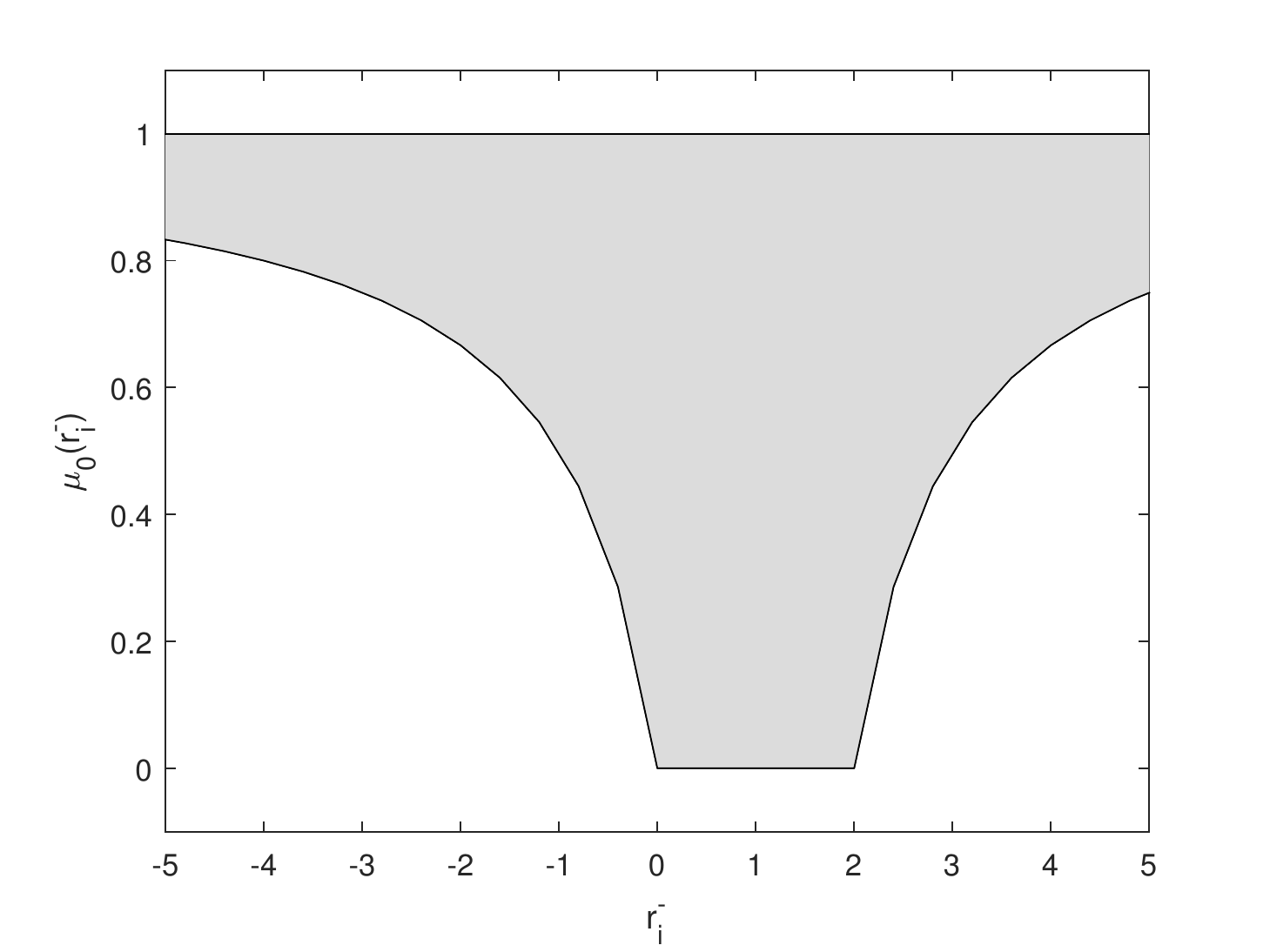} \\
			\textbf{a}&\textbf{b}
		\end{tabular}
		\caption{Region for data-bounded approximation using \eqref{reconst2} at $x_{i-\frac{1}{2}}:$ \textbf{(a)} Any approximation, \textbf{(b)} WENO approximation.}
		\label{fig:LDBregion}	
	\end{figure}
				
It is clear from the carried out analysis that any weight $\beta$ and $\mu$ lying within data-bounded shaded region given in Figure \ref{fig:RDBregion}(a) and \ref{fig:LDBregion}(a) will yield  data-bounded approximations $\hat{v}_{i + \frac{1}{2}}$ and $\hat{v}_{i - \frac{1}{2}}$ respectively. 
\\

\section{High order data-bounded WENO approximations: DB-WENO}\label{sec-3}
In this section high order data-bounded WENO approximation of the function $v(x)$ at the cell-interface $x_{i+\frac{1}{2}}$ is discussed. Note that weights $\beta_0$ and $\mu_0$ defined in \eqref{U.B-K} and \eqref{L.B-J} respectively give WENO approximation. However to ensure high order accuracy in WENO 
approximations at cell interfaces in sooth solution region, the non-linear weights $\beta_0$ and $\mu_0$ must attain ideal non-linear weight \cite{parvin2021new},

\subsection{High order DB-WENO interpolations}\label{subsec-3.1}
Here the data values are $v_j$ i.e., the point values of the function $v(x)$ at points $x_j$'s. Define the weights, 
\begin{equation}\label{weightsR}
\beta_0=\min(1/4,|K|),~K=\min\left(1, \frac{sgn(r_i^{+})}{r_i^{+}-1}\right)
\end{equation}
\begin{equation}\label{weightsL}
\mu_0=\max\left(3/4,\min\left(\frac{2-r_i^{-}}{1-r_i^{-}}, \frac{-r_i^{-}}{1-r_i^{-}}\right)\right)
\end{equation}
Note that, by the convexity property of weights, the other non-linear weights in \eqref{reconst1} and \eqref{reconst2} are $\beta_1=1-\beta_0$ and $~\mu_1=1-\mu_0$ respectively. From Figure \ref{DBweight} it can be seen that these non-linear weights $\beta_0$ and $\mu_0$ lies inside the data-bounded region of the approximations \eqref{reconst1} and \eqref{reconst2} respectively. Therefore they ensure the data-boundedness of approximations using \eqref{reconst1} and \eqref{reconst2} along with third order accuracy in smooth data region. Other possible DB-weights are given in Appendix A, which lies inside the data-bound region of \eqref{reconst1} and \eqref{reconst2}.
\begin{figure}[htb!] 
\begin{tabular}{cc}
	\centering
	\includegraphics[scale=0.5]{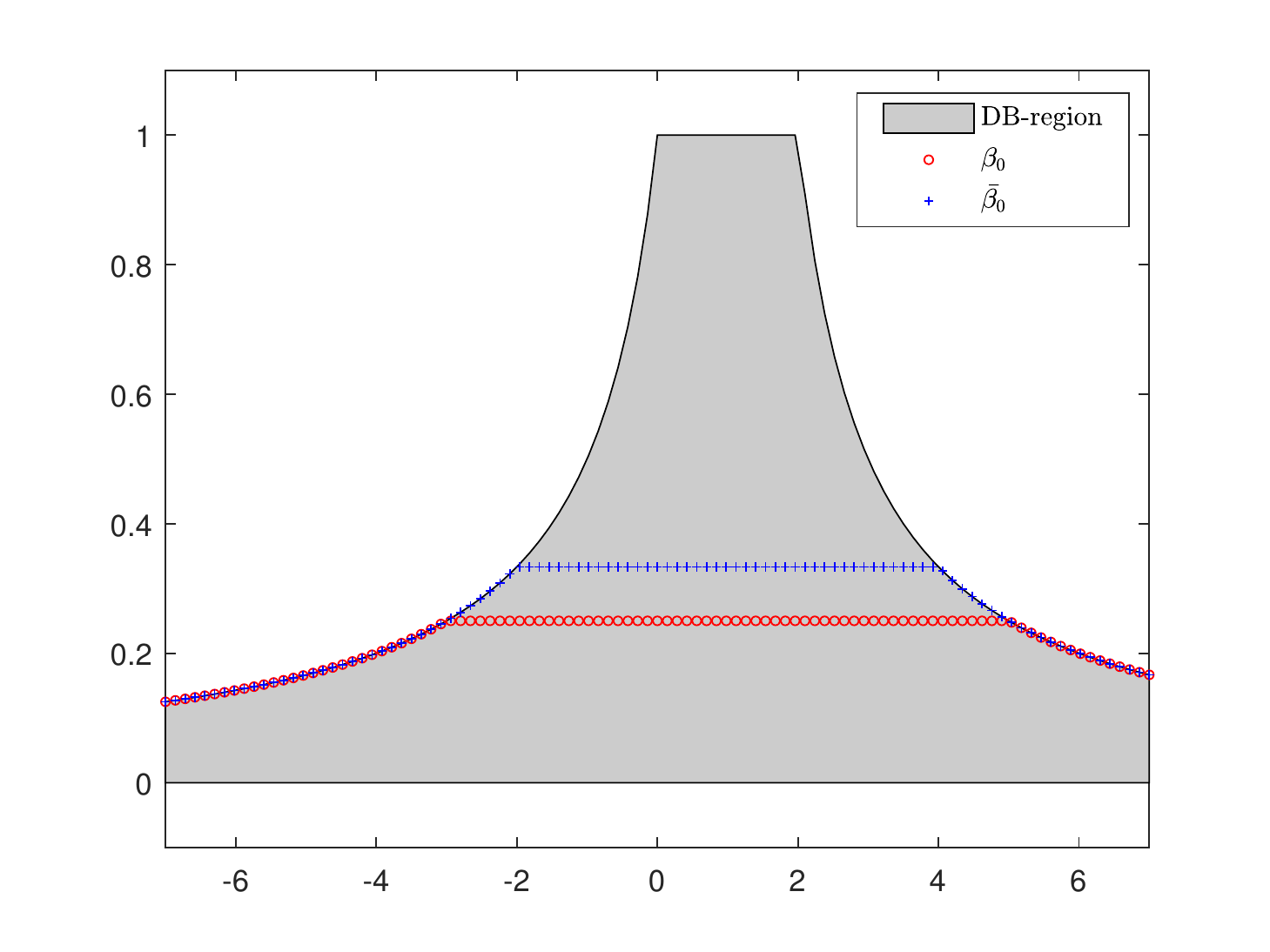} &
	\includegraphics[scale=0.5]{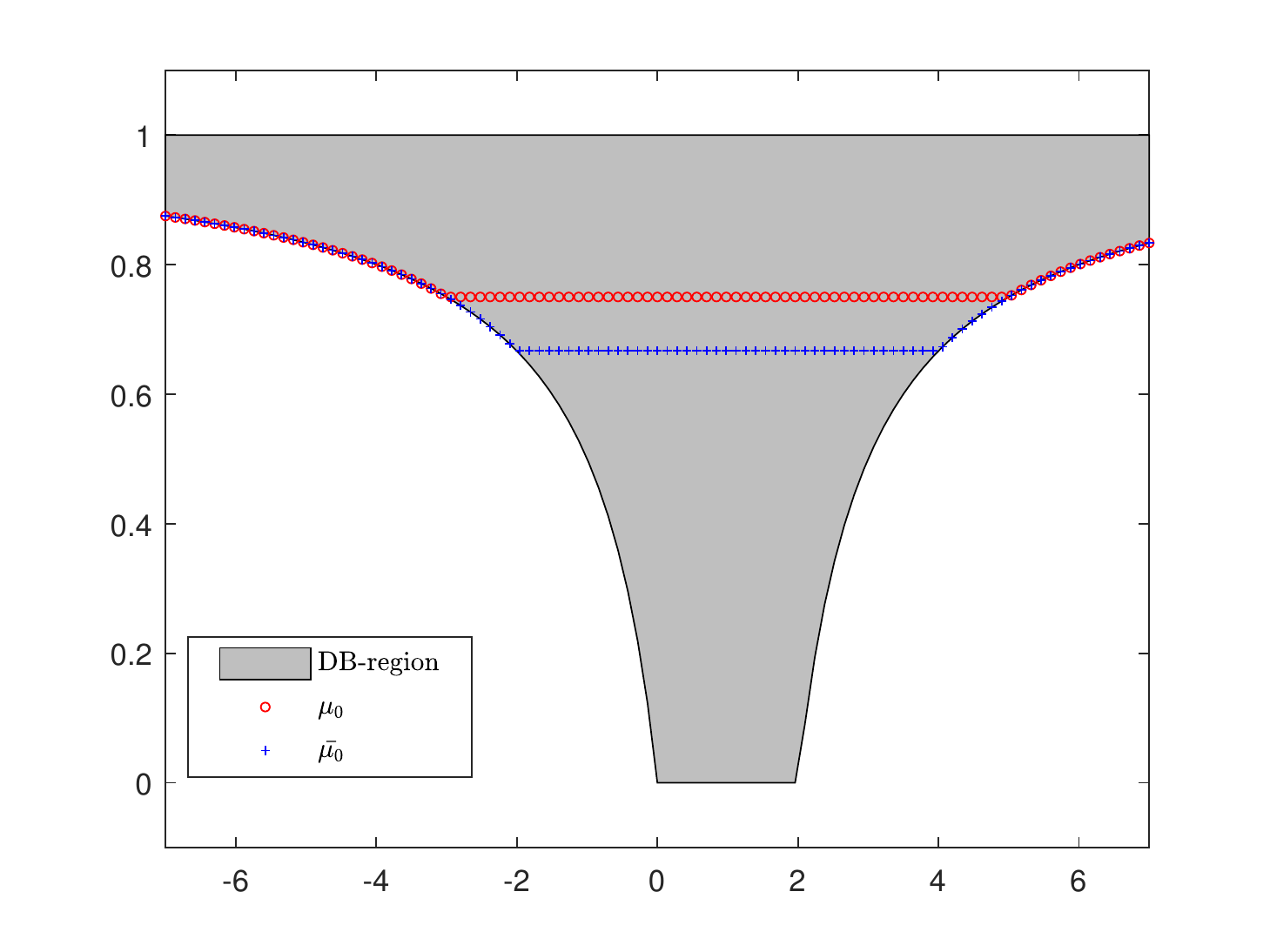} \\
	\textbf{(a)}&\textbf{(b)}
\end{tabular}
\caption{Proposed weights inside data-bound region}
\label{DBweight}	
\end{figure}

\subsubsection{Third order interpolation:} The DB-weights $\beta_0$ defined in \eqref{weightsR}) yield data-bounded interpolation $\hat{v}_{i+\frac{1}{2}}$ in \eqref{reconst1}. Note that in smooth data region $r_{i}^{\pm} \approx 1 $, therefore $\beta_0 =\frac{1}{4}$ and \eqref{reconst1} reduces to 
\begin{equation}
\hat v_{i+\frac{1}{2}}= \frac{1}{4}\left( \frac{3}{2}v_i-\frac{1}{2}v_{i-1}\right) + \frac{3}{4} \left( \frac{1}{2}v_i+\frac{1}{2}v_{i+1}\right),
\end{equation}
A simple Taylor series analysis shows
\begin{equation}
\hat{v}_{i+\frac{1}{2}} = v(x_{i + \frac{1}{2}}) + O(\Delta x^3)
\end{equation}

\subsubsection{Fourth order interpolation:} Fourth order accurate approximation at $x_{i+\frac{1}{2}}$ using four-point stencil $S(i)=\{x_{i-1},x_{i},x_{i+1},x_{i+2}\}$  can be written as \cite{janett2019novel} 
\begin{equation}\label{reconst4}
	\hat {\hat v}_{i+\frac{1}{2}}= \frac{1}{2}~\hat \hat v^{0}_{i+\frac{1}{2}} + \frac{1}{2}~\hat \hat v^{1}_{i+\frac{1}{2}},
\end{equation}
The approximations  $\hat {\hat{v}}^0~\&~\hat{\hat{v}}^1$ for the stencils $S_p(i):=\{x_{i+p-1},x_{i+p},x_{i+p+1}\},~(p=0,1)$ are of the form
\begin{equation}\label{subreconst4}
	\begin{aligned}
		\hat {\hat{v}}^{0}_{i+\frac{1}{2}}=\beta_0~ \left(\frac{3}{2}v_i-\frac{1}{2}v_{i-1}\right)+\beta_1~\left(\frac{1}{2}v_i+\frac{1}{2}v_{i+1}\right)\\
		\hat{\hat{v}}^{1}_{i+\frac{1}{2}}=\mu_0~\left(\frac{1}{2}v_i+\frac{1}{2}v_{i+1}\right)+\mu_1~\left(\frac{3}{2}v_{i+1}-\frac{1}{2}v_{i+2}\right)
	\end{aligned}
\end{equation}
The approximation \eqref{subreconst4} using the weights $\beta_0$ and $\mu_0$  defined in equations \eqref{weightsR} and \eqref{weightsL} ensure for data-boundedness of the interpolation \eqref{reconst4}. In smooth solution region $\beta_0 =\frac{1}{4}, \mu_0 = \frac{3}{4}$ thus, 	
\begin{eqnarray}\label{subreconst4acc}
\hat{\hat{v}}_{i+\frac{1}{2}}&=&\frac{1}{2}\left(\frac{1}{4} \left(\frac{3}{2}v_i-\frac{1}{2}v_{i-1}\right)+\frac{3}{4}~\left(\frac{1}{2}v_i+\frac{1}{2}v_{i+1}\right)\right)\\\nonumber
& & +\frac{1}{2}\left(\frac{3}{4}\left(\frac{1}{2}v_i+\frac{1}{2}v_{i+1}\right)+\frac{1}{4}\left(\frac{3}{2}v_{i+1}-\frac{1}{2}v_{i+2}\right)\right).
\end{eqnarray}	
It using simple Taylor series shows fourth-order accuracy in smooth data case. i.e., 
\begin{equation}
	\hat{\hat{v}}_{i+\frac{1}{2}} = v(x_{i + \frac{1}{2}}) + O(\Delta x^4)
\end{equation}

\subsection{High order DB-WENO reconstructions}\label{subsec-3.2}
Here the data values are  $\overline v_j$ i.e., the cell-average values of the function $v(x)$ in cell $[x_j, x_{j+1}]$. Define the following non-linear DB-weights $\overline \beta_0$ and $\overline \mu_0$ inside the data-bounded region in Figure \ref{DBweight}(a) and \ref{DBweight}(b).
\begin{equation}\label{Re-weightsR}
\overline \beta_0=\min(1/3,|K|),~K=\min\left(1, \frac{sgn(\overline r_i^{+})}{\overline r_i^{+}-1}\right)
\end{equation}
\begin{equation}\label{Re-weightsL}
\overline \mu_0=\max\left(2/3,\min\left(\frac{2-\overline r_i^{-}}{1-\overline r_i^{-}}, \frac{-\overline r_i^{-}}{1-\overline r_i^{-}}\right)\right),
\end{equation}
where $\overline r_i^{\pm}=\frac{\Delta_{\mp}\overline v_i}{\Delta_{\pm}\overline v_i} \in \mathbb{R}\cup\{\pm \infty\}$ and $\overline v_i=\frac{1}{\Delta x}\int_{x_{i -\frac{1}{2}}}^{x_{i + \frac{1}{2}}}v(\xi) d\xi.$\\

\subsubsection{Third order reconstruction:} The use of the DB-weight $\overline \beta_0$ define in \eqref{Re-weightsR} yields a data-bounded WENO reconstruction i.e.
\begin{equation}\label{Rreconst1}
\hat{v}_{i+\frac{1}{2}}=\overline \beta_0~ \left(\frac{3}{2}\overline v_i-\frac{1}{2}\overline v_{i-1}\right)+(1-\overline \beta_0)~\left(\frac{1}{2}\overline v_i+\frac{1}{2}\overline v_{i+1}\right). 
\end{equation}
Note that WENO reconstruction \eqref{Rreconst1} is third order \cite{shu2009high} i.e.,
$$\hat{v}_{i+\frac{1}{2}} = v(x_{i + \frac{1}{2}}) + O(\Delta x^3).$$

\subsubsection{Fourth order reconstruction:} The use of the DB-weights $\overline \beta_0$ and $\overline \mu_0$  defined in equations \eqref{Re-weightsR} and \eqref{Re-weightsL}) ensure for data-boundedness of the WENO reconstruction  is of the form
\begin{eqnarray}\label{Rreconst4}
\hat{\hat v}_{i+\frac{1}{2}}&=&	\frac{1}{2}\left[\overline \beta_0 \left(\frac{3}{2}\overline v_i-\frac{1}{2}\overline v_{i-1}\right)+\overline \beta_1\left(\frac{1}{2}\overline v_i+\frac{1}{2}\overline v_{i+1}\right)\right]\\ \nonumber 
& & + \frac{1}{2}\left[\overline \mu_0\left(\frac{1}{2}\overline v_i+\frac{1}{2}\overline v_{i+1}\right)+\overline  \mu_1\left(\frac{3}{2}\overline v_{i+1}-\frac{1}{2}\overline v_{i+2}\right)\right],
\end{eqnarray}
where $\overline \beta_1=1-\overline \beta_0$ and $~\overline \mu_1=1-\overline \mu_0.$ It can be shown using Taylor series expansion that the WENO reconstruction \eqref{Rreconst4} is fourth order accurate  i.e, 
$$\hat{\hat{v}}_{i+\frac{1}{2}} = v(x_{i + \frac{1}{2}}) + O(\Delta x^4).$$

\subsection{Computational verifications}\label{subsec-3.3}
In this section the data boundedness of the DB-WENO approximations (both interpolations and reconstructions) are verified. The accuracy of the DB-WENO approximations are also shown in tabular form. In each test case, a finite domain $x\in [-1,1]$ is considered with periodic boundary conditions. The following name convention is used through out this section.
\begin{itemize}
\item DBI-WENO3 \& DBR-WENO3 represents the result of data-bounded third order interpolation \eqref{reconst1} and reconstruction \eqref{Rreconst1} obtained by the DB-weights $\beta_0,~\overline \beta_0$ respectively.
\item DBI-WENO4 \& DBR-WENO4 represents the result of data-bounded fourth order interpolation \eqref{reconst4} and reconstruction \eqref{Rreconst4} obtained by the DB-weights $(\beta_0,\mu_0)$ \& $(\overline \beta_0,\overline \mu_0)$ respectively.
\item Lagrange3 \& Lagrange4 represents the results obtained by third order and fourth order Lagrange interpolations at $x_{i+\frac{1}{2}}$. 
\end{itemize}
\subsubsection{Test for accuracy of DB-WENO approximation at $x_{i+\frac{1}{2}}$:} 
We consider the data generated from smooth function
\begin{equation}\label{Ex1}
v(x)=sin(\pi x),~~x\in[-1,1]
\end{equation}
The error in approximating the interface values and the corresponding convergence rates are shown in Tables \ref{table:DBApprox-3} \& \ref{table:DBApprox-4}. Table \ref{table:DBApprox-3} shows that DBI-WENO3 and DBR-WENO3 give third order accuracy in both $L^1$ and $L^{\infty}$ norm. The fourth order convergence rate of DBI-WENO4 and DBR-WENO4 in the norms $L^1$ and $L^{\infty}$ are cleary visible by the Table \ref{table:DBApprox-4}. 
\begin{table}[htb!]
\centering
\begin{tabular}{|c|c|c|c|c|}
	\hline N  & DBI-WENO3  &  Rate &    DBI-WENO3   & Rate \\ 
	& $L^\infty$ error  &   &    $L^1$ error    &  \\  
	\hline 40 & 2.82050e-04 &  - & 3.74372e-04 & - \\ 
	\hline 80 & 3.26558e-05 &  3.11 & 4.24272e-05 & 3.14 \\ 
	\hline 160 & 3.93012e-06 &  3.05 & 5.05406e-06 & 3.07 \\ 
	\hline 320 & 4.82089e-07 &  3.03 & 6.16857e-07 & 3.03 \\ 
	\hline 640 & 5.96973e-08 &  3.01 & 7.61965e-08 & 3.02 \\ 
	\hline 1280 & 7.42722e-09 &  3.01 & 9.46827e-09 & 3.01 \\ 
	\hline 
	& &  &  &\\
	\hline N  & DBR-WENO3  &  Rate &    DBR-WENO3   & Rate \\ 
	& $L^\infty$ error  &   &    $L^1$ error    &  \\  
	\hline 40 & 3.75767e-04 &  - & 4.98765e-04 & - \\ 
	\hline 80 & 4.35328e-05 &  3.11 & 5.65588e-05 & 3.14 \\ 
	\hline 160 & 5.23991e-06 &  3.05 & 6.73843e-06 & 3.07 \\ 
	\hline 320 & 6.42779e-07 &  3.03 & 8.22467e-07 & 3.03 \\ 
	\hline 640 & 7.95962e-08 &  3.01 & 1.01595e-07 & 3.02 \\ 
	\hline 1280 & 9.90297e-09 &  3.01 & 1.26243e-08 & 3.01 \\ 
	\hline 
\end{tabular} 
\caption {Error and convergence rate of third order DB-WENO approximations at $x_{i+\frac{1}{2}}$.}
\label{table:DBApprox-3}
\end{table}	

\begin{table}[htb!]
\centering
\begin{tabular}{|c|c|c|c|c|}
	\hline N  & DBI-WENO4  &  Rate &    DBI-WENO4   & Rate \\ 
	& $L^\infty$ error  &   &    $L^1$ error    &  \\  
	\hline 40 & 1.74786e-05 &  - & 2.23558e-05 & - \\ 
	\hline 80 & 9.86319e-07 &  4.15 & 1.25718e-06 & 4.15 \\ 
	\hline 160 & 5.86064e-08 &  4.07 & 7.46397e-08 & 4.07 \\ 
	\hline 320 & 3.57197e-09 &  4.04 & 4.54827e-09 & 4.04 \\ 
	\hline 640 & 2.20467e-10 &  4.02 & 2.80712e-10 & 4.02 \\ 
	\hline 1280 & 1.36933e-11 &  4.01 & 1.74349e-11 & 4.01 \\ 
	\hline 
	& &  &  &\\
	\hline N  & DBR-WENO4  &  Rate &    DBR-WENO4   & Rate \\ 
	& $L^\infty$ error  &   &    $L^1$ error    &  \\  
	\hline 40 & 2.48342e-05 &  - & 3.17639e-05 & - \\ 
	\hline 80 & 1.40244e-06 &  4.15 & 1.78757e-06 & 4.15 \\ 
	\hline 160 & 8.33466e-08 &  4.07 & 1.06148e-07 & 4.07 \\ 
	\hline 320 & 5.08007e-09 &  4.04 & 6.46856e-09 & 4.04 \\ 
	\hline 640 & 3.13573e-10 &  4.02 & 3.99239e-10 & 4.02 \\ 
	\hline 1280 & 1.94549e-11 &  4.01 & 2.47363e-11 & 4.01 \\ 
	\hline 
\end{tabular} 
\caption {Error and convergence rate of fourth order DB-WENO approximations at $x_{i+\frac{1}{2}}$.}
\label{table:DBApprox-4}
\end{table}	

\subsubsection{Test for data-boundedness:}
We now verify the data-boundedness of DB-WENO interpolations and reconstructions by two test functions.
\begin{itemize}
\item[\textbf{Test 1:}] Consider the data generated from a smooth runge function
\begin{equation}\label{Ex2}
v(x)=\frac{1}{1+25 x^2},~~~~x\in [-1,1]
\end{equation} 
\item[\textbf{Test 2:}] Consider the data generated from a discontinuous function
\begin{equation}\label{Ex3}
v(x)=\left\{\begin{array}{cc}
1 & |x|\leq 0.3,\\
0 & else.
\end{array}\right.
\end{equation}	
\end{itemize}
In Figure \ref{DBApprox-smooth} and Figure \ref{DBApprox-discont} underlying function data of \eqref{Ex2} and \eqref{Ex3} respectively is given along with the third and forth order reconstructed values at $x_{i+\frac{1}{2}}$ using DBI-WENO, DBR-WENO and Lagrange approximation. Figure \ref{DBApprox-smooth} and Figure \ref{DBApprox-discont} verify that third and fourth-order Lagrange interpolations are not data-bounded approximations, while DBI-WENO, DBR-WENO i.e., WENO interpolations and WENO reconstructions using proposed DB-weights are data-bounded approximations.

\begin{figure}[htb!] 
%	\begin{tabular}{cc}
%		\hspace{-0.8cm}
	\includegraphics[scale=0.55]{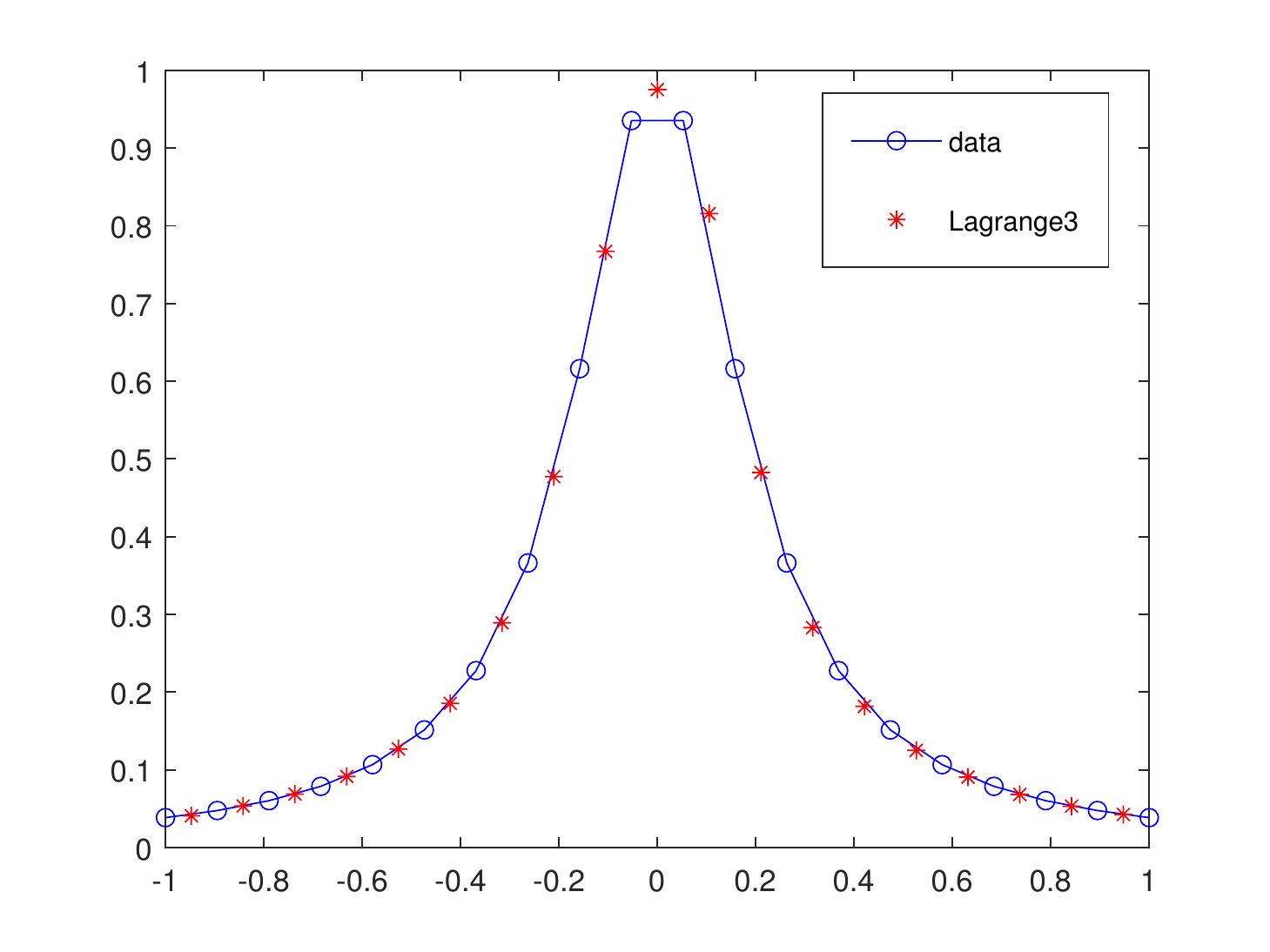}
	\includegraphics[scale=0.55]{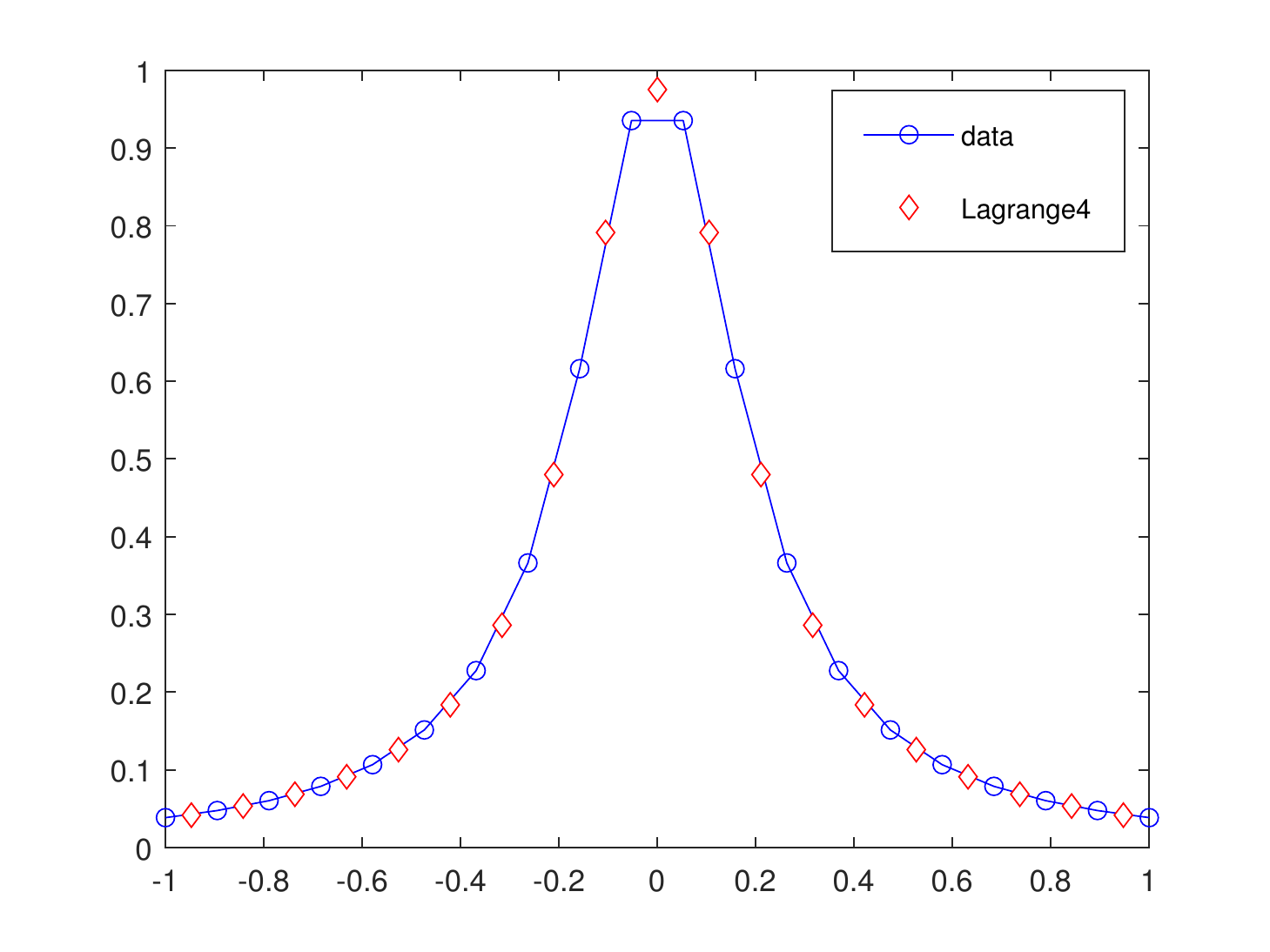}
	\includegraphics[scale=0.55]{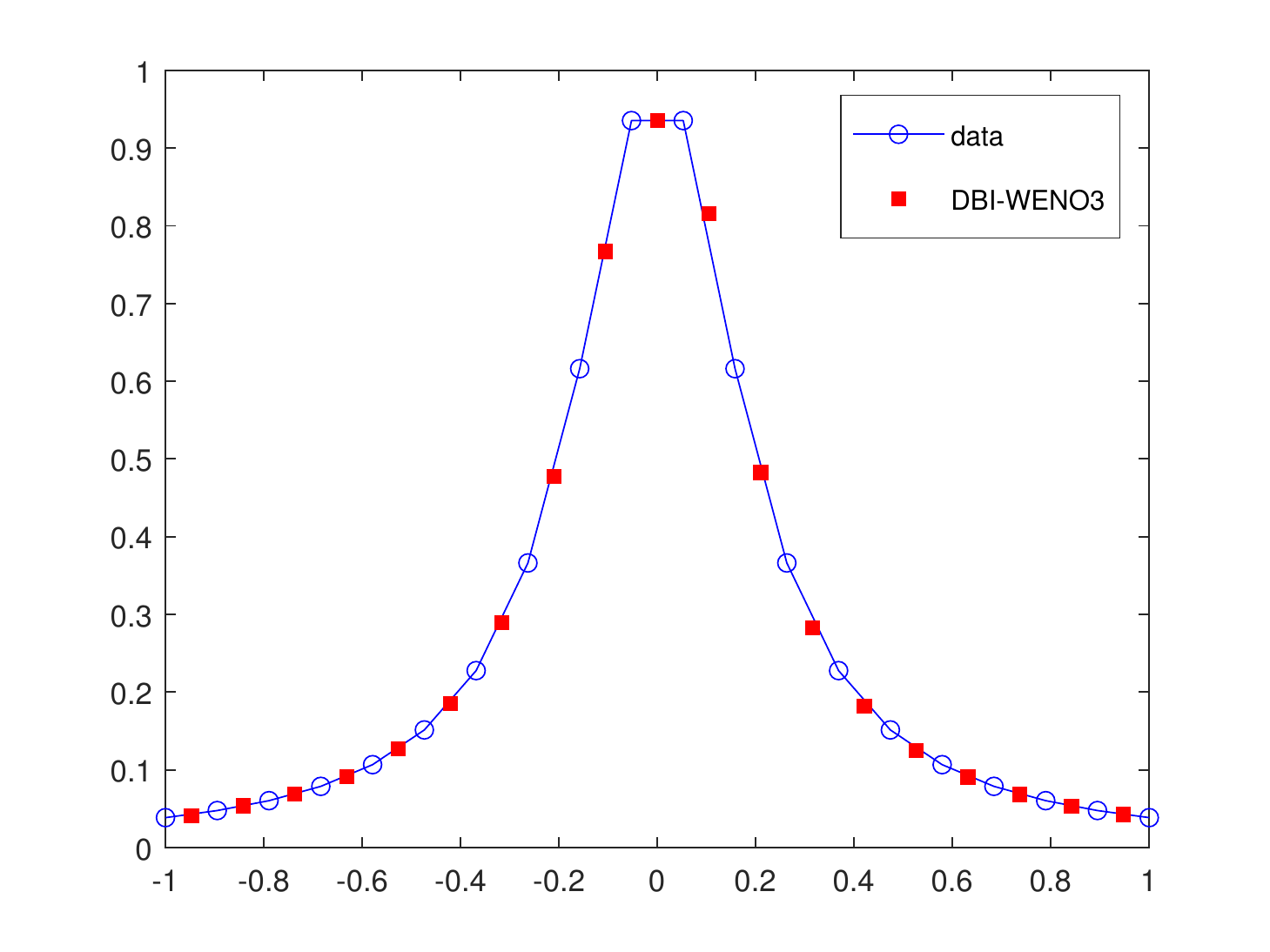} 
	\includegraphics[scale=0.55]{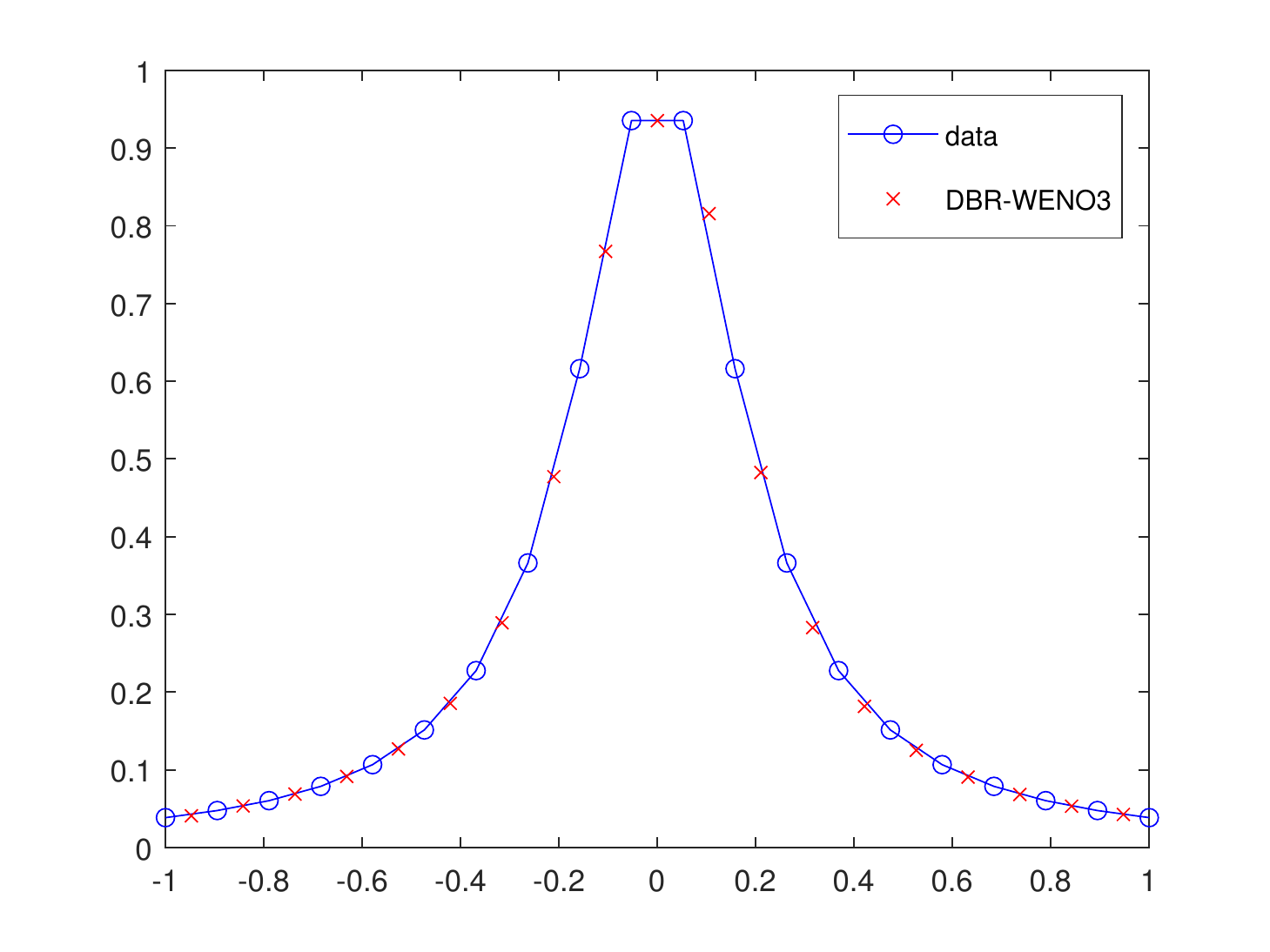}
	\includegraphics[scale=0.55]{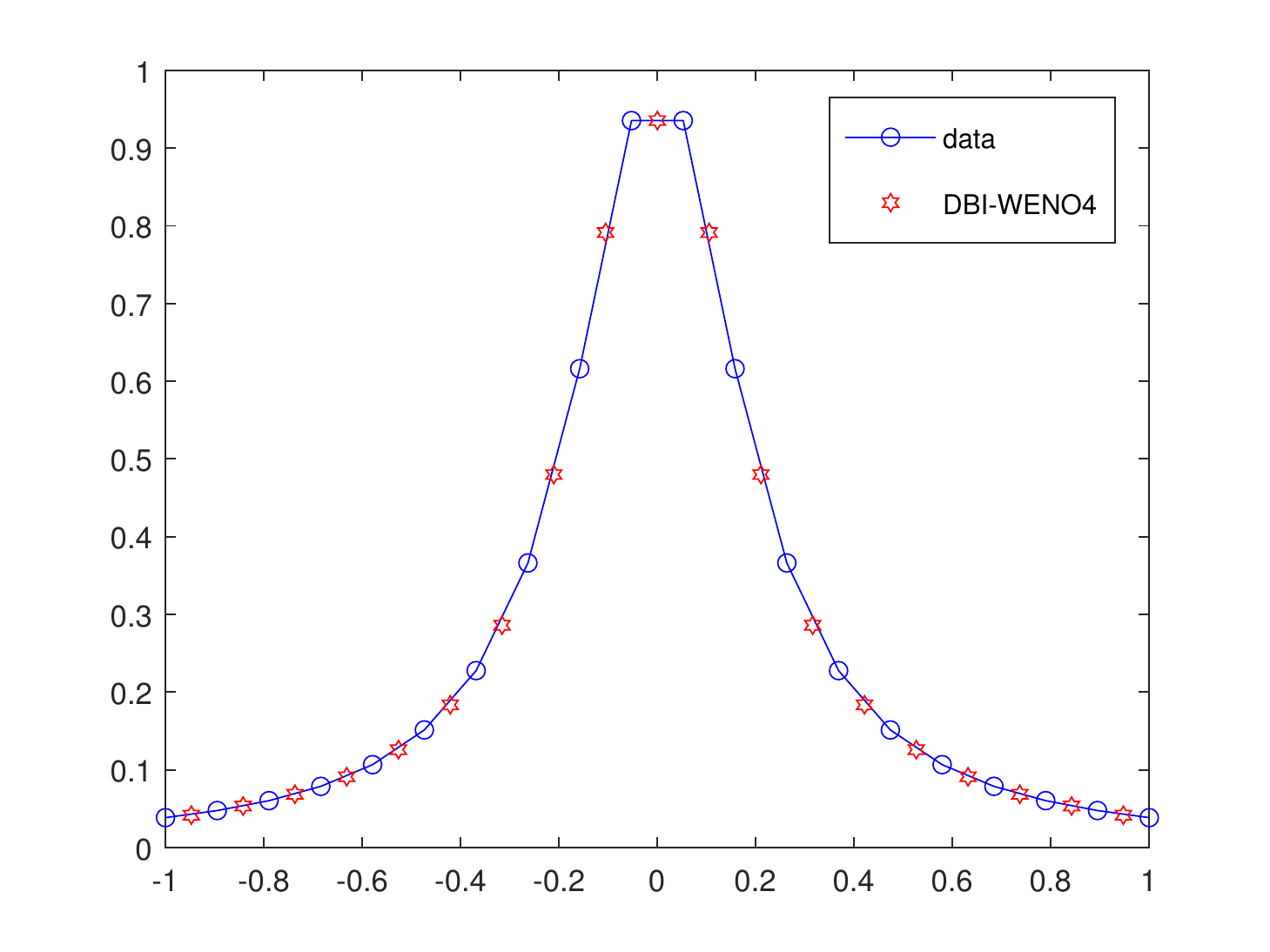} 
	\includegraphics[scale=0.55]{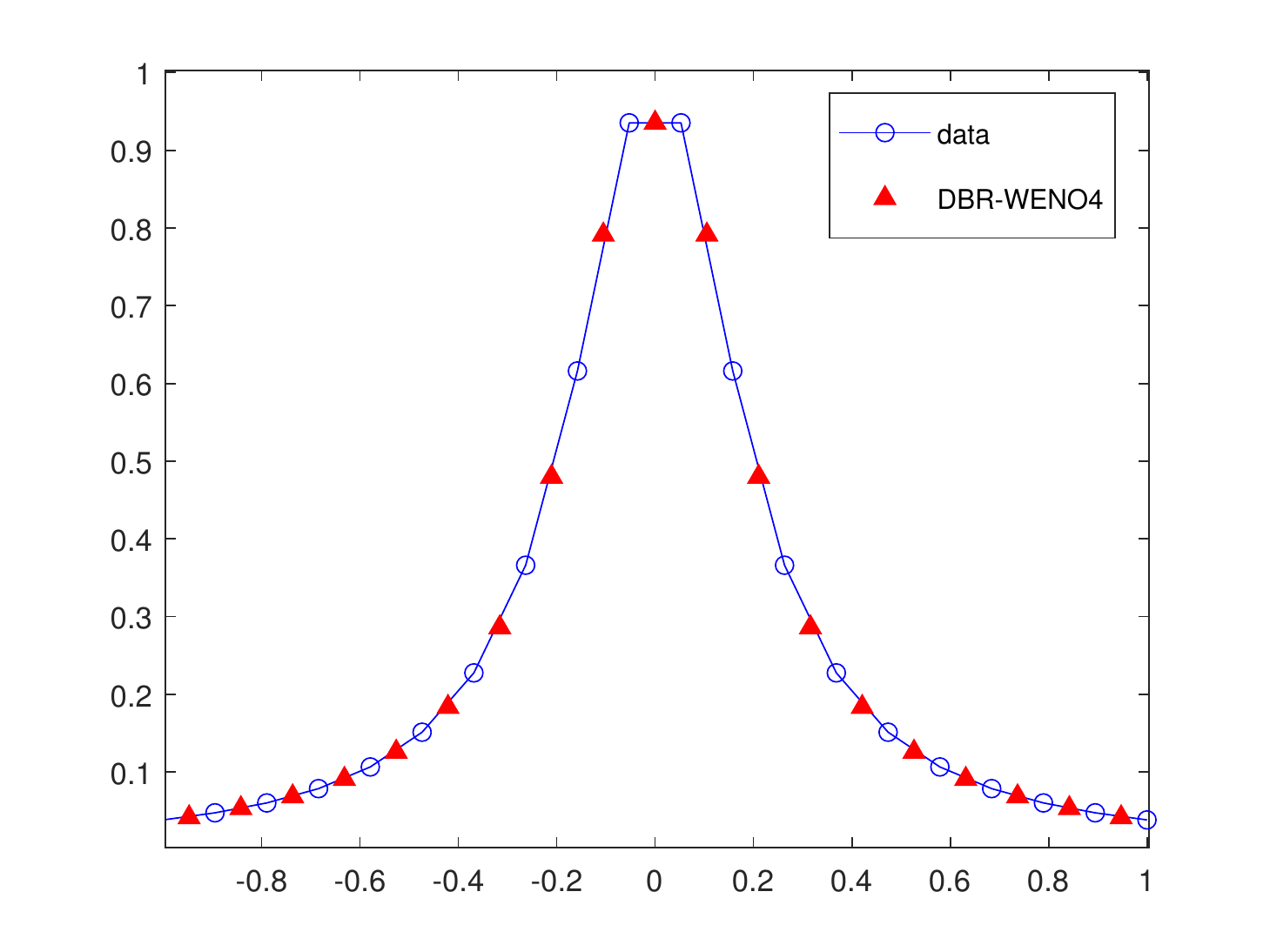}
%		\textbf{a}&\textbf{b}
%	\end{tabular}
\caption {Data-boundedness of third and fourth order DB-WENO approximations for smooth runge function \eqref{Ex2} with 20 grid points.}
\label{DBApprox-smooth}	
\end{figure} 
\begin{figure}[htb!] 
%	\begin{tabular}{cc}
%		\hspace{-1.2cm}
	\includegraphics[scale=0.55]{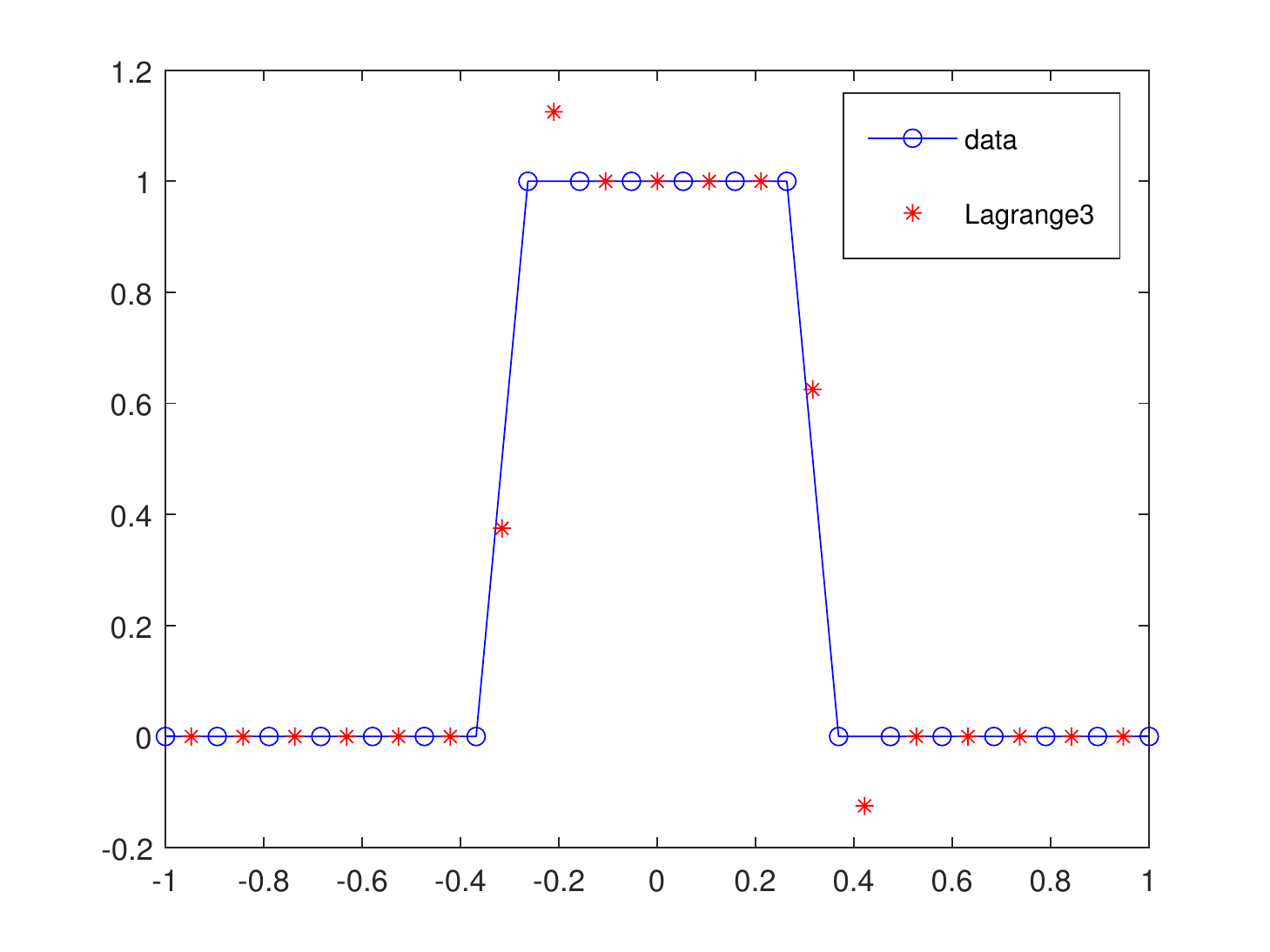}
	\includegraphics[scale=0.55]{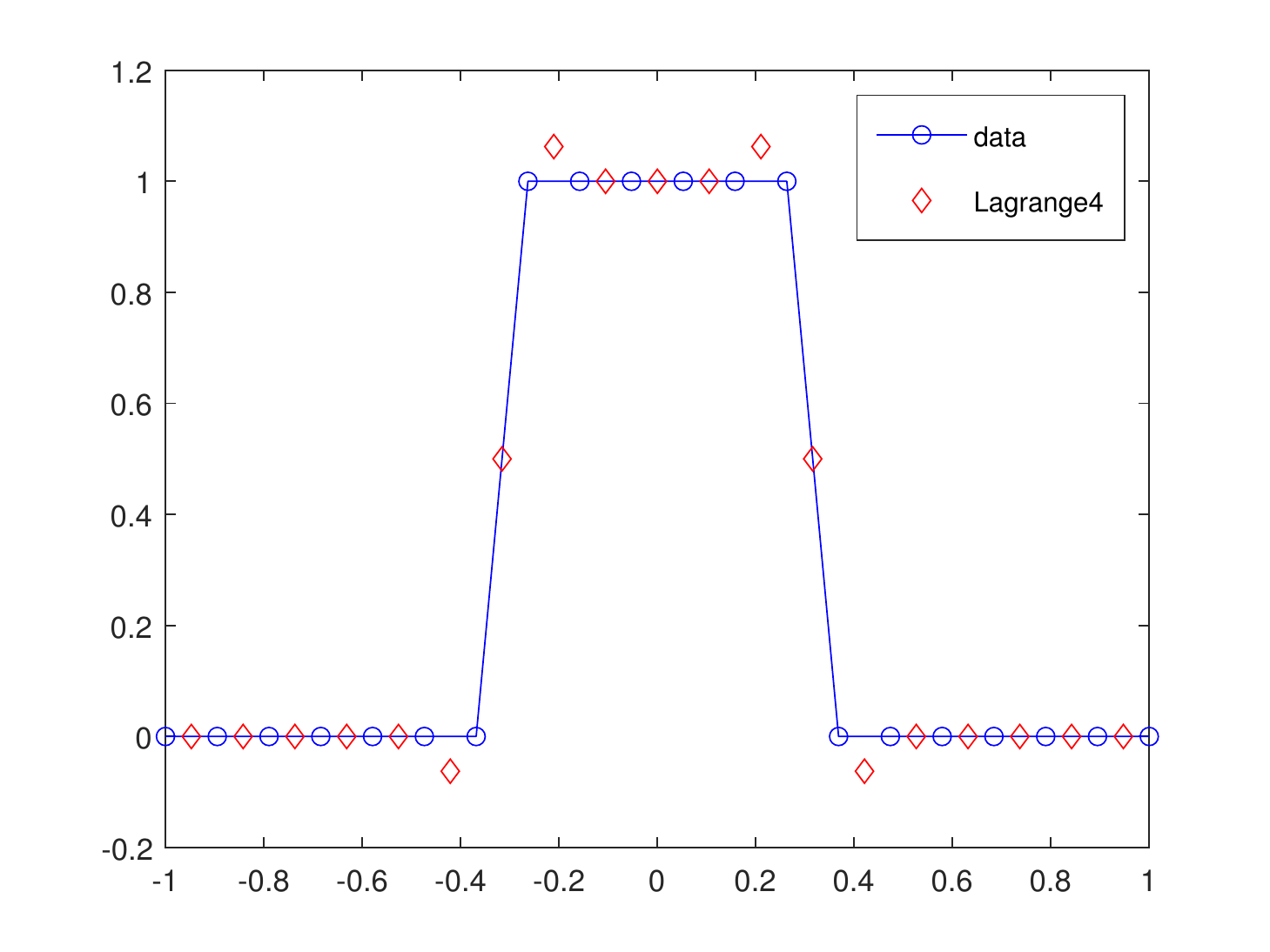}
	\includegraphics[scale=0.55]{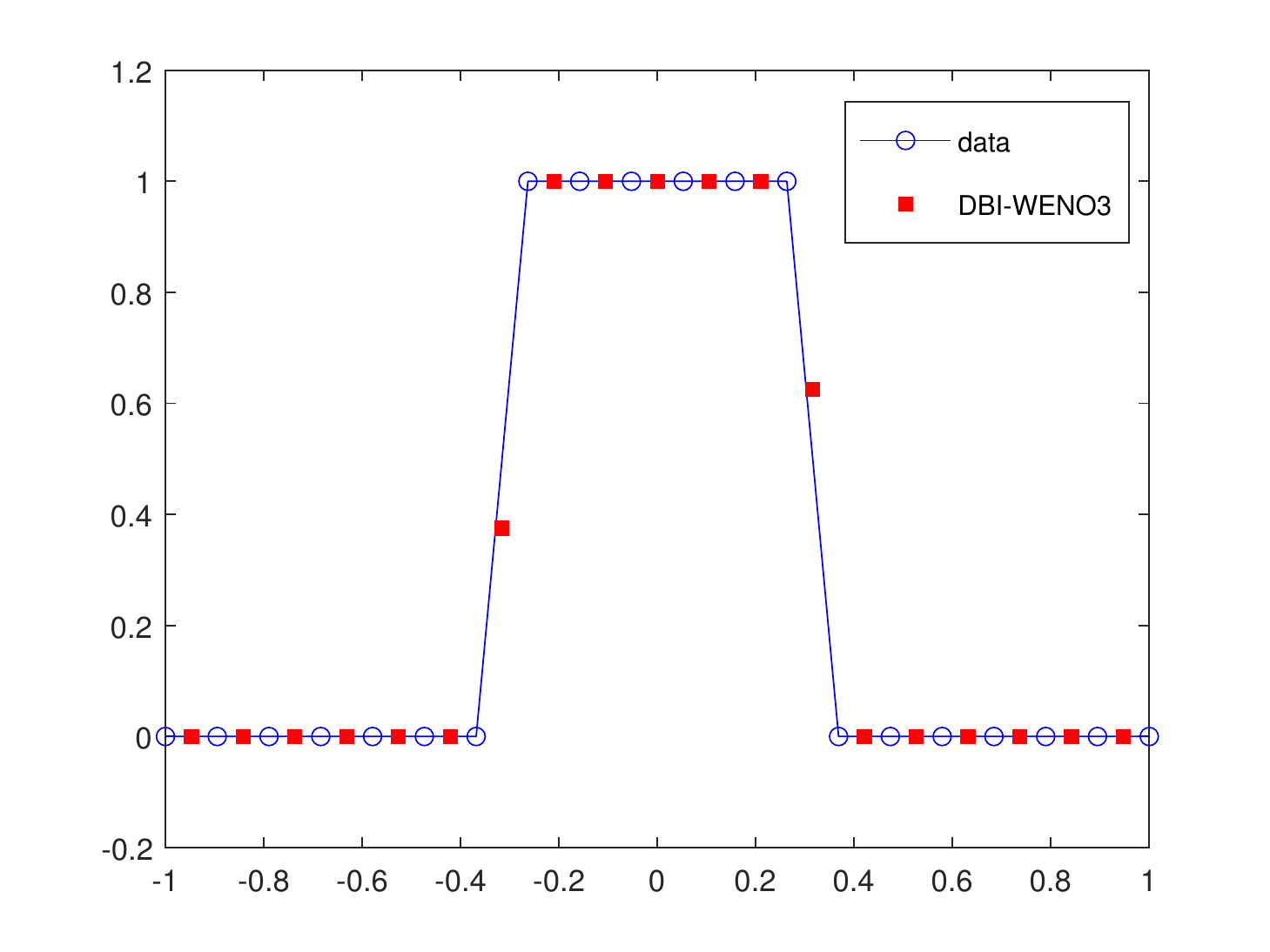}
	\includegraphics[scale=0.55]{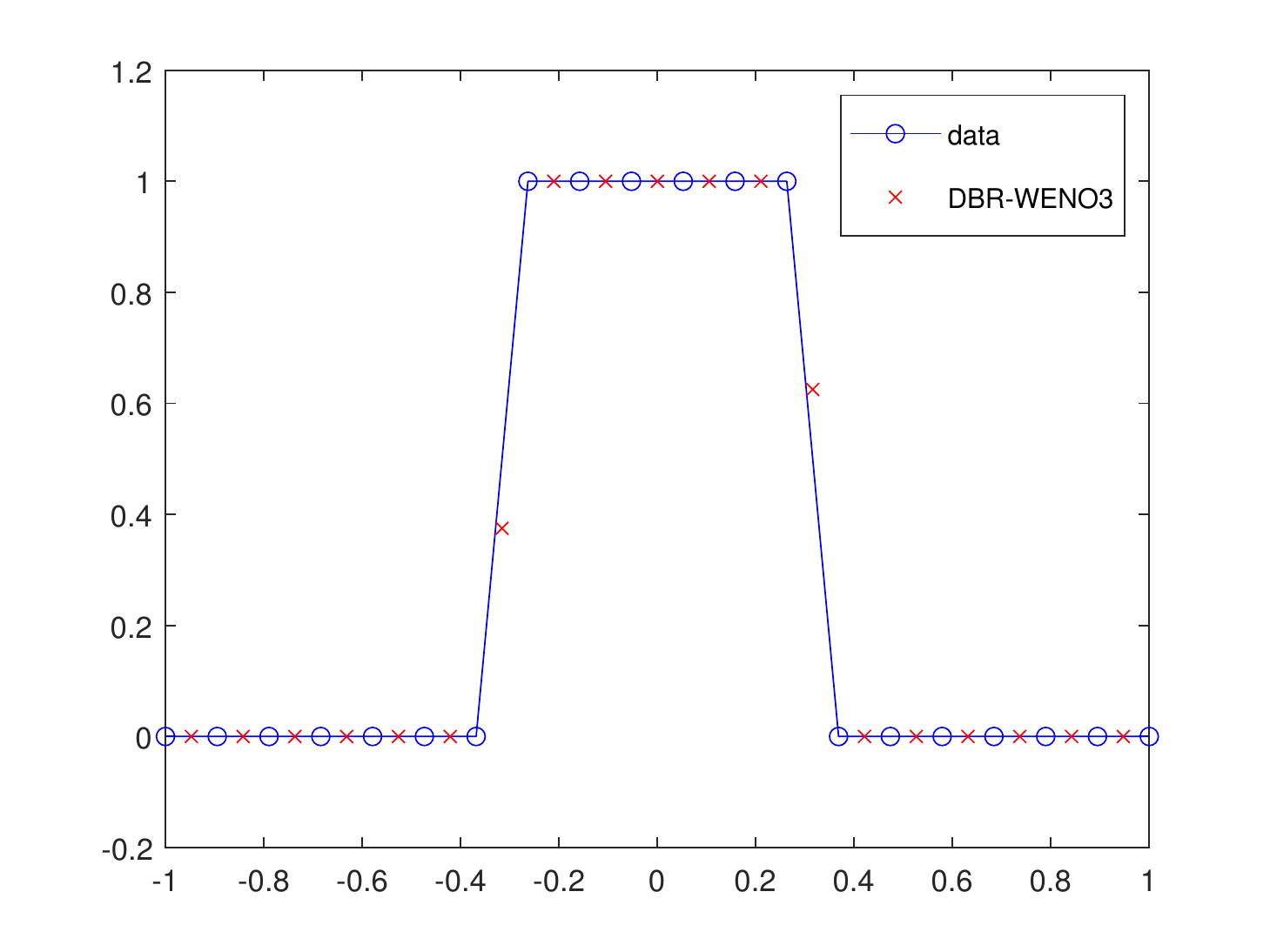}
	\includegraphics[scale=0.55]{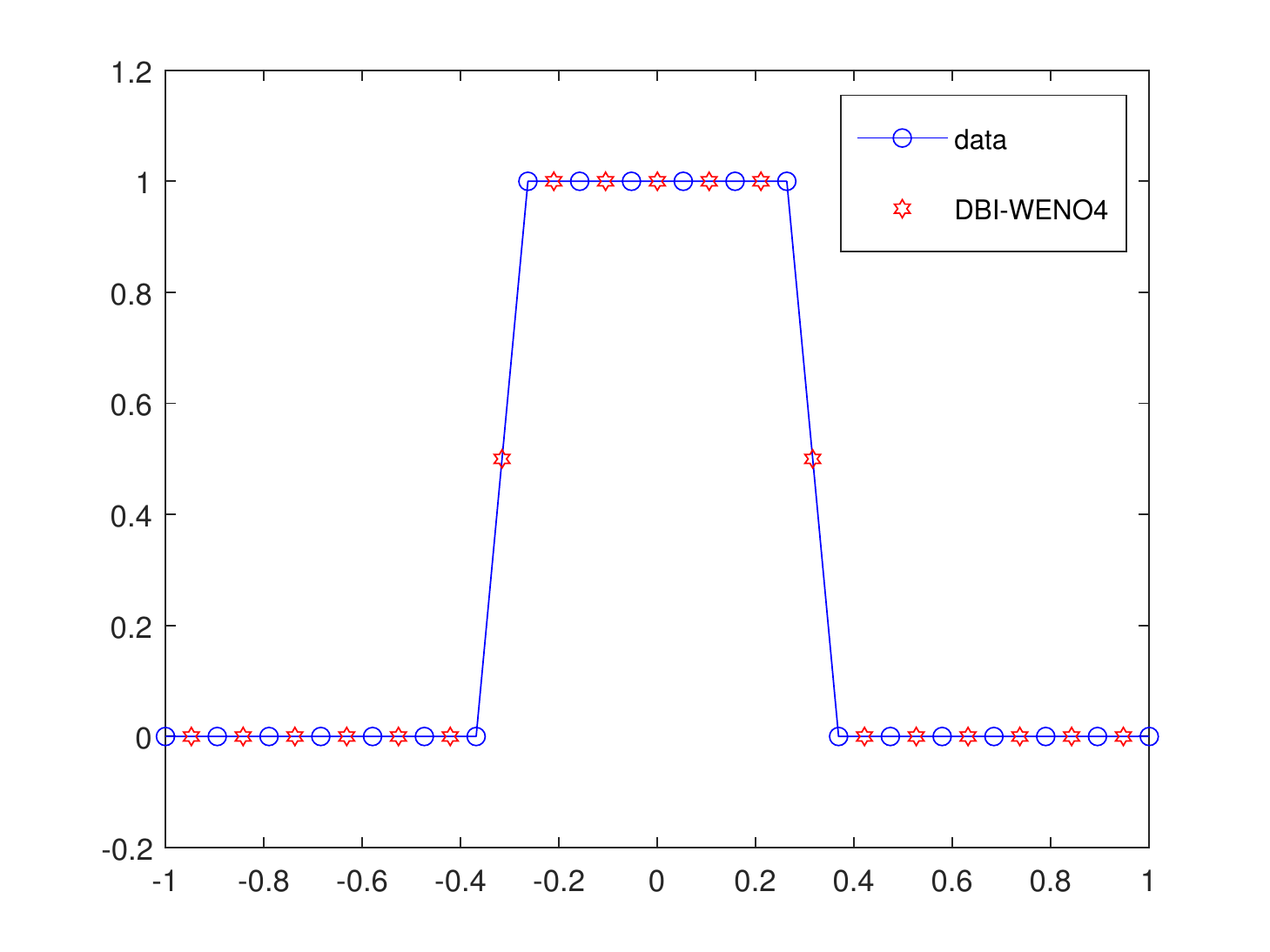}
	\includegraphics[scale=0.55]{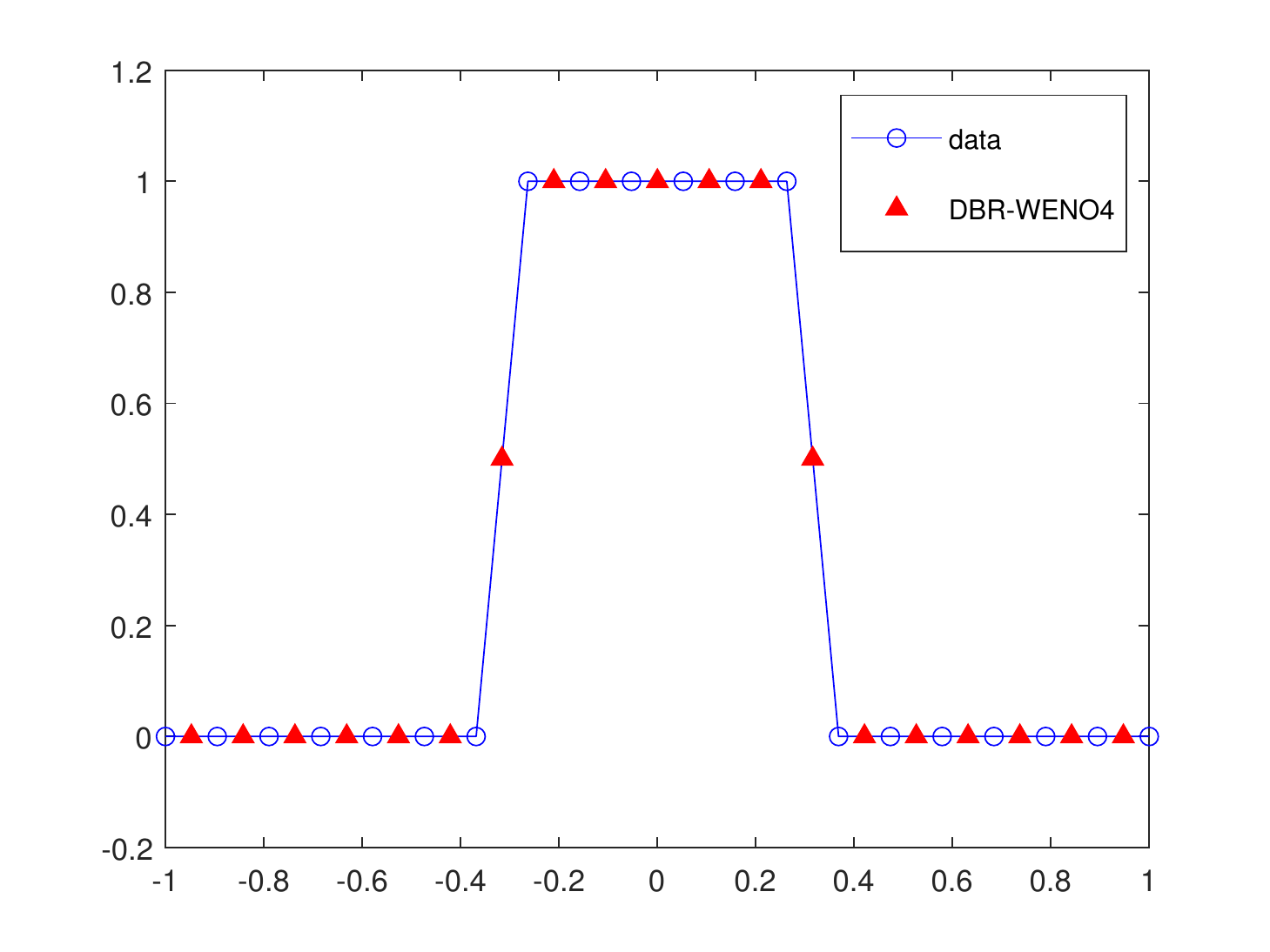}
%		\textbf{a}&\textbf{b}
%	\end{tabular}
\caption{Data-boundedness of third and fourth order DB-WENO approximations for discontinuous function \eqref{Ex3} with 20 grid points.}	
\label{DBApprox-discont}	
\end{figure}

\section{Application of DB-WENO in WENO3 scheme of HCL}\label{sec-4}
The semi-discretized conservative scheme for the one-dimensional hyperbolic conservation laws $u_t+f(u)_x=0$ is as follows
\begin{gather}\label{semi-disc1}
\frac{du_i(t)}{dt}=-\frac{1}{\Delta x}\left(\hat{f}_{i+\frac{1}{2}}-\hat{f}_{i-\frac{1}{2}}\right)
\end{gather}
In the third-order WENO (WENO3) scheme \cite{shu1998essentially} the numerical flux $\hat{f}_{i+\frac{1}{2}}$ is of the form 
\begin{equation}\label{weno3-scheme}
\hat{f}_{i+\frac{1}{2}}=\sum_{l=0}^{1}\omega_l \hat{f}^{l}_{i+\frac{1}{2}},
\end{equation}
where non-linear weights $\omega_l$ satisfy the following convexity property
\begin{equation}
\sum_{l=0}^1 \omega_l=1,\; \omega_l\geq 0. 
\end{equation}
and the expression of $\hat{f}^{l}_{i+\frac{1}{2}}$ are as follows
\begin{equation}\label{weno3}
\begin{aligned}
	\hat{f}^{0}_{i+\frac{1}{2}}=\frac{3}{2}f_i-\frac{1}{2}f_{i-1}\\
	\hat{f}^{1}_{i+\frac{1}{2}}=\frac{1}{2}f_i+\frac{1}{2}f_{i+1}.
\end{aligned}
\end{equation}
Here the point values of the physical flux of hyperbolic conservation laws i.e., $f(u_i)=f_i$ are considered as the given data values and for maintaining the accuracy of WENO3 schemes, the ideal or linear weights are $d_0=\frac{1}{3},~~d_1=\frac{2}{3}.$
\subsection{DB-weights for WENO3 scheme}\label{subsec-4.1}
In authors previous work \cite{parvin2021new} the following characterization is given for non-oscillatory weights for third order WENO schemes:
\begin{itemize}
\item [i.]Data-bound conditions: $\omega_0$ must be inside the data bounded region Figure \ref{fig:RDBregion}-b.
\item [ii.]Non-oscillatory conditions:
\begin{itemize}
	\item[(a)] for $r_i^{+}=0$,~$\omega_0=1$ and $\omega_1=0$ (Upwind only flux).
	\item[(b)] for $r_i^{+}$ away from $[-1,3]$,~$\omega_0=0$ and $\omega_1=1$ (Centered only flux).	
\end{itemize} 
\item [iii.] Third order Accuracy Condition: for smooth region of solution (which corresponds to $r_i^{+}\approx 1$) ideal(linear) weights should be choosen as $\omega_0=\frac{1}{3}, \; \omega_1=\frac{2}{3}.$ \\
\end{itemize}
Therefore, the use of data bounded weight $\omega_0=\beta_0$ of  \eqref{weightsR} in the numerical flux \eqref{weno3-scheme} results into oscillatory WENO3 scheme. Some DB-weights are defined in \cite{parvin2021new} which lie inside data-bound region Figure \ref{fig:DBweno3}, in order to make WENO3 scheme data-bounded.
\begin{subequations}\label{DBweno3}
\begin{equation}\label{DBweno3_1}
	\omega^1_0=\frac{1}{3}+\frac{2}{3}\left(1-\frac{3|r_i|}{2|r_i|+1}\right)
\end{equation}
\begin{equation}\label{DBweno3_2}
	\omega^2_0=\frac{1}{3}+\frac{2}{3}\left(1-\min\left(\frac{2|r_i|}{1+|r_i|}, \frac{3}{2}\right)\right)
\end{equation}
\begin{equation}\label{DBweno3_3}
	\omega^{k}_{0,5}= \frac{1}{3}+\frac{2}{3}\left(1-\displaystyle \min\left(k|r_i|,\max\left(1, \frac{3|r_i|}{2|r_i|+k}\right)\right)\right),~1.5\le k\le 2
\end{equation}	
\end{subequations}
\begin{figure}[htb!] 
\begin{tabular}{cc}
	\hspace{-1 cm}
	\includegraphics[scale=0.5]{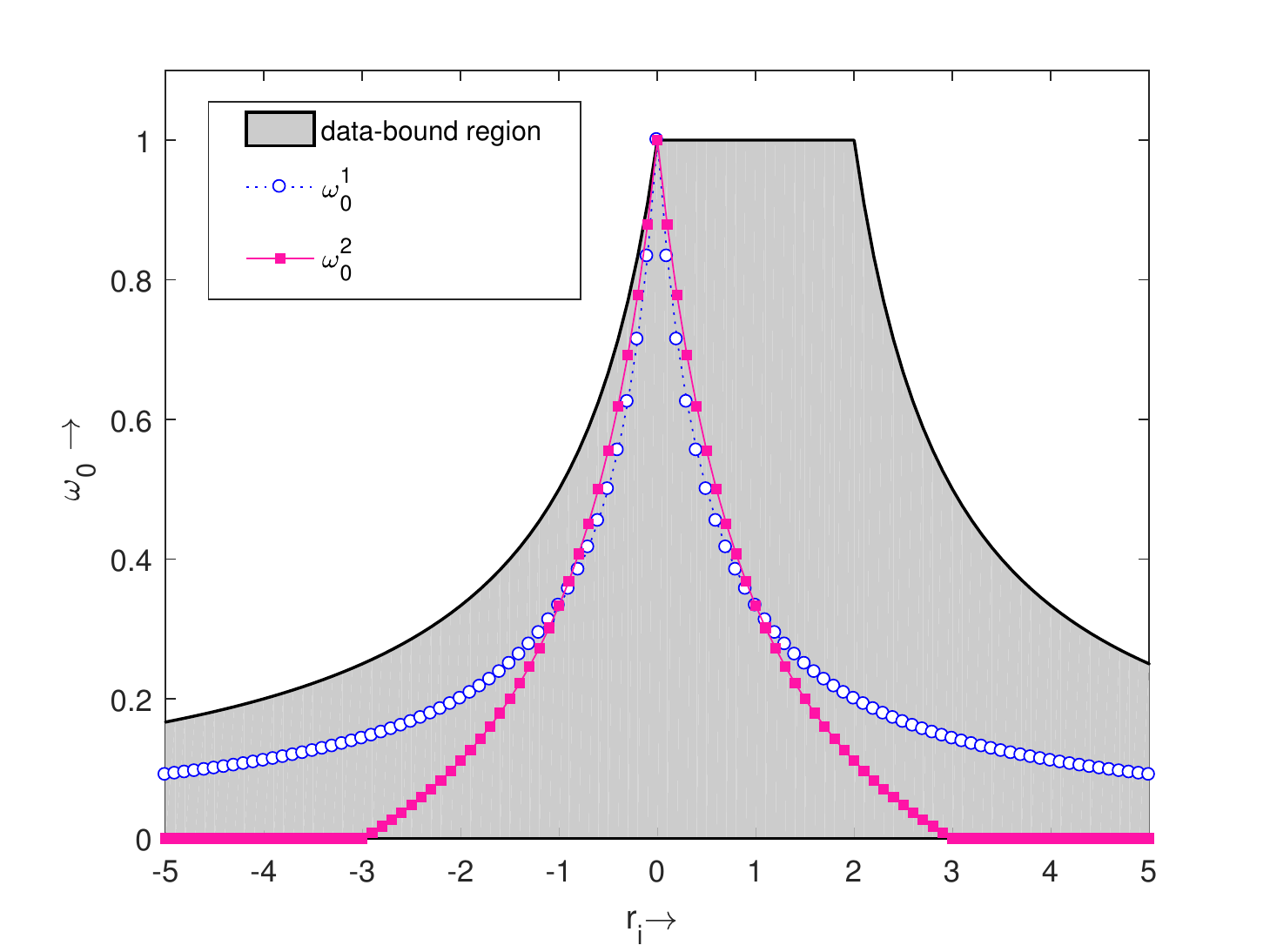} &
	\includegraphics[scale=0.5]{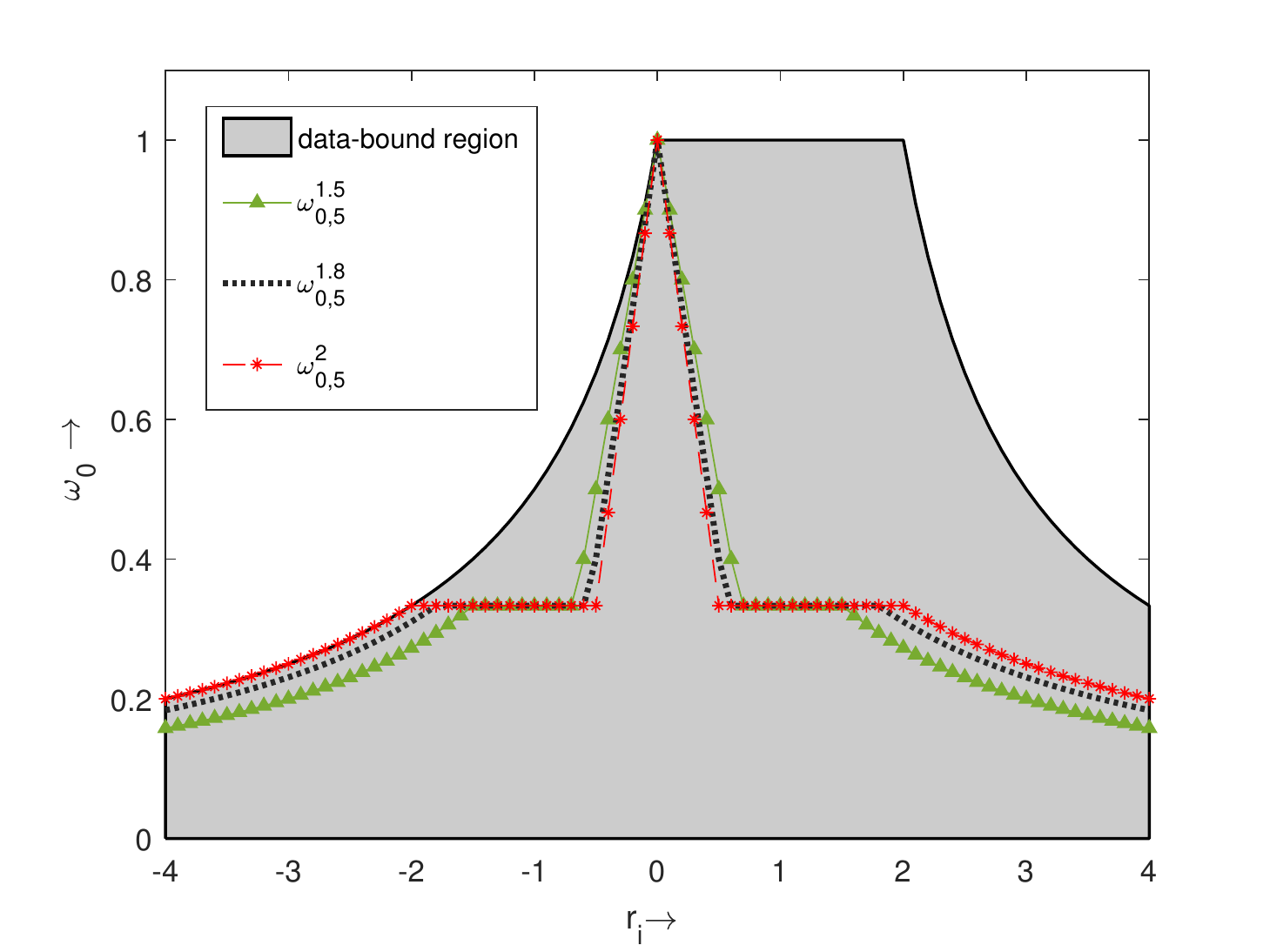} \\
	\textbf{(a)}&\textbf{(b)}
\end{tabular}
\caption{Inside the data-bounded region (a) weights \eqref{DBweno3_1},\eqref{DBweno3_2}~(b)~weights \eqref{DBweno3_3} with different $k$.}
\label{fig:DBweno3}	
\end{figure}
In Figure \ref{fig:DBweno3}, it can be seen that all these weights $\omega^1_0,\omega^2_0,\omega^{k}_{0,5}$ lies inside the data-bound region of third order  WENO approximation and as a result they ensure the data-boundedness of the third order WENO3 scheme. We refer interested to our previous work \cite{parvin2021new} for  numerical results of WENO3 scheme with the proposed weights $\omega^1_0,\omega^2_0,\omega^{k}_{0,5}$.
\section{Other Applications of DB-WENO}\label{sec-5}
\begin{itemize}
	\item The data-bounded WENO approximations can be applied for several real-life PDE based problems which contains both strong discontinuities and complex smooth solution features. For example, DB-WENO approximations can be applied to the the Convection dominated \cite{shu2009high} and Hamilton-Jacobi problems \cite{osher1991high} to develop data-bounded schemes for those problems.  
	\item It is important to note that WENO method is actually an approximation procedure, not directly related to PDEs, hence the DB-WENO procedure can also be used in many non-PDE applications, including computer vision and image processing.  	
\end{itemize}
\section{Conclusion and Future work}\label{sec-6}
Conditions on non-linear weights for constructing data-bounded polynomial approximation are obtained. Based on these bounds, weights are constructed to ensure the data-boundedness of third and fourth order WENO approximations (interpolation and reconstructions) for smooth solution data. Numerical results to show the accuracy and boundedness of the WENO approximations are also given. The fourth order data-bounded WENO approximation which is described in the section \ref{sec-3}, can be used to make data-bounded fourth order WENO (WENO4) scheme, but for that detailed charaterization of the structure of WENO weights is required. This work is under investigation and can be reported separately in future.

{\bf Acknowledgment:} Authors acknowledge the Science and Engineering Board, New Delhi, India for providing necessary financial support through funded projects File No. EMR/2016/000394.
%\pagebreak
%\bibliographystyle{unsrt}
%\bibliography{ref}

\textbf{Appendix A: Other Non-linear DB-weights }
\begin{itemize}
	\item[$\bullet$] Some proposed non-linear DB-weights for the region \eqref{U.B-K} in Corrolary \ref{corrolary1}.
	\begin{itemize}
		\item[i.)] 
		\begin{equation}\label{weightsR1}
		\beta_0^{\eta}=\eta K
		\end{equation}
		where $0\le \eta \le 1$ and $K$ is defined in equation \eqref{U.B-K}.
		\item [ii.)] 
		\begin{equation}\label{weightsR3}
		\beta_0^{2}=\left\{\begin{array}{cc} 1/4 & if~r_{i}^{+}\in [-3,5], \\
		\frac{8}{3{r_i^{+}}^2+5} &  if~r_{i}^{+}<-3,\\
		\frac{5}{3{r_i^{+}}^2-55}&  if~r_{i}^{+}>5.\end{array}\right.
		\end{equation}
	\end{itemize}
	
	\item[$\bullet$] Some proposed non-linear DB-weights for the region \eqref{L.B-J} in Corrolary \ref{corrolary2}.
	\begin{itemize}
		\item[i.)] 
		\begin{equation}\label{weightsL1}
		\mu_0^{\eta}=1-\eta (1-J)
		\end{equation}
		where $0\le \eta \le 1$ and $J$ is defined in equation \eqref{L.B-J}.
		\item [ii.)] 
		\begin{equation}\label{weightsL3}
		\mu_0^{2}=\left\{\begin{array}{cc} 3/4 & if~r_{i}^{-}\in [-3,5], \\
		\frac{3({r_i^{-}}^2-1)}{3{r_i^{-}}^2+5} &  if~r_{i}^{-}<-3,\\
		\frac{3({r_i^{-}}^2-20)}{3{r_i^{-}}^2-55}&  if~r_{i}^{-}>5.\end{array}\right.
		\end{equation}
	\end{itemize}
	
	Note that by the convexity property of weights, the other non-linear weights $\beta_1^{\eta,2}=1-\beta_0^{\eta,2},\&~\mu_1^{\eta,2}=1-\mu_0^{\eta,2}.$ 
	From Figure \ref{fig:DBweights} it can be seen that different non-linear DB-weights defined by the above equations are lies inside the data-bound region.
\end{itemize}
\begin{figure}[htb!] 
	\begin{tabular}{cc}
		% 		\hspace{-1.2cm}
		\includegraphics[scale=0.55]{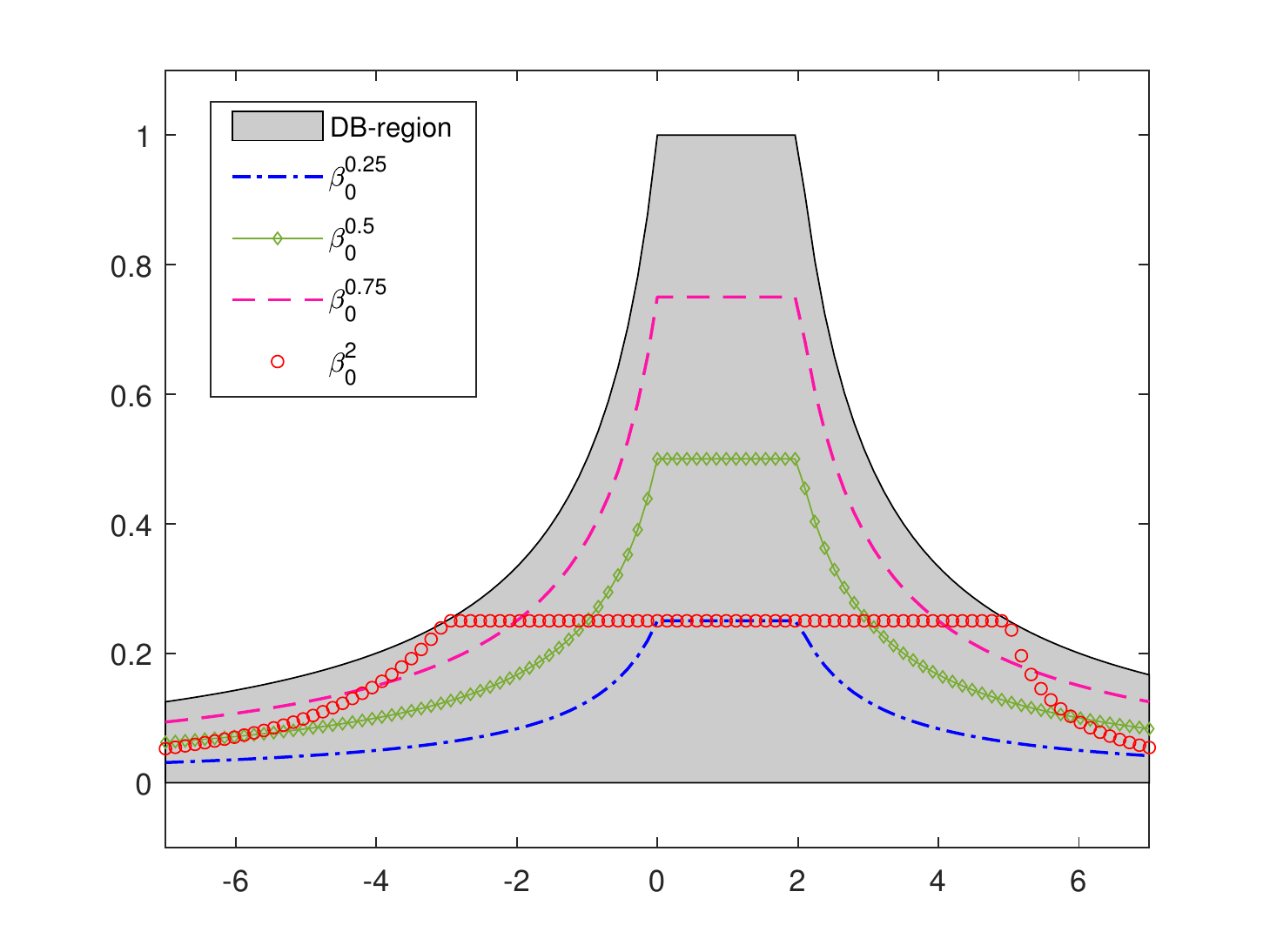} &
		\includegraphics[scale=0.55]{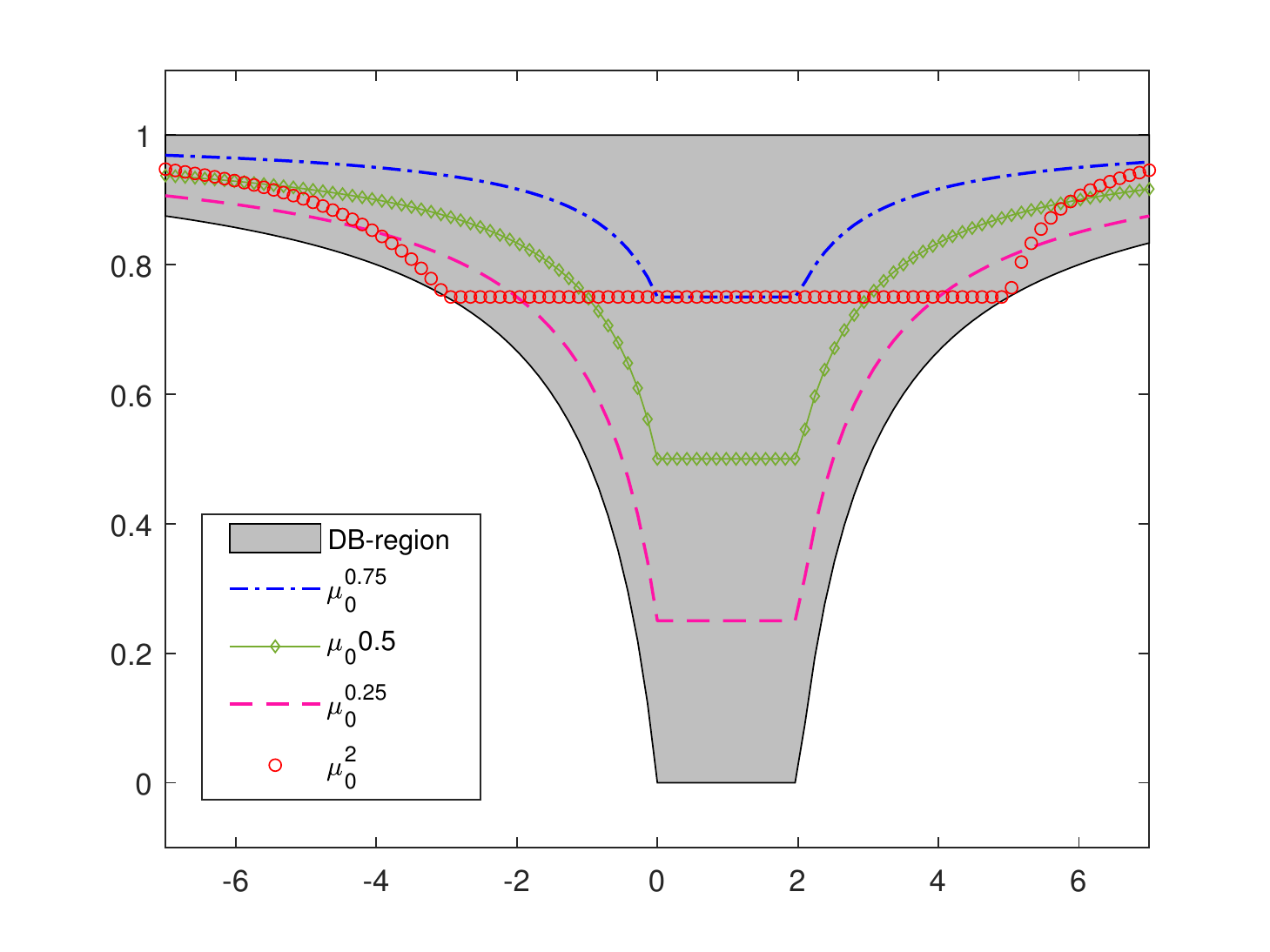} \\
		\textbf{a}&\textbf{b}
	\end{tabular}
	\caption{The proposed DB-weights inside data-bound region }	
	\label{fig:DBweights}	
\end{figure}

\end{document}